\documentclass[aap]{imsart}
\usepackage{amsthm,amsmath,amssymb,natbib,srcltx,bbm,a4wide,xargs,mathabx,xr,color}
\usepackage[shortlabels]{enumitem}
\RequirePackage[colorlinks,citecolor=blue,urlcolor=blue]{hyperref}

\startlocaldefs
\newcommand{\norm}[1]{\lVert #1 \rVert}
\def\ie{i.e.}
\def\iid{i.i.d.}

\def\tbZ{\widetilde{\bZ}}
\def\coormap{T}
\newcommand\seqspace[1]{#1^\Zset}
\newcommand\zeroseq[1]{{\boldsymbol{0}}_{{\mathsf#1}^\Zset}}
\def\tailmeasure{{\boldsymbol{\nu}}}

\def\equallaw{\stackrel{d}{=}}
\def\setE{\mathsf{E}}
\def\setEzero{\setE\setminus\{\zero{\setE}\}}
\def\setF{\mathsf{F}}
\def\setFzero{\setF\setminus\{\zero{\setF}\}}
\def\setS{\mathsf{S}}
\def\Nset{\mathbb N}
\def\Zset{\mathbb Z}
\def\Rset{\mathbb R}
\def\esp{\mathbb E}
\def\pr{\mathbb P}

\def\mca{\mathcal{A}}

\def\mce{\mathcal{E}}
\def\mcf{\mathcal{F}}
\def\mcm{\mathcal{M}}
\def\mcn{\mathcal{N}}

\def\mct{\mathcal{T}}
\def\rme{\mathrm{e}}
\def\rmd{\mathrm{d}}

\def\bQ{\boldsymbol{Q}}

\def\bx{\boldsymbol{x}}
\def\bs{\boldsymbol{s}}
\def\bX{\boldsymbol{X}}
\def\by{\boldsymbol{y}}
\def\bY{\boldsymbol{Y}}

\def\bZ{\boldsymbol{Z}}
\def\bP{\boldsymbol{P}}
\def\bxi{\boldsymbol\xi}

\def\bTheta{\boldsymbol\Theta}

\def\bszero{\boldsymbol{0}}
\def\nullmeasure{\boldsymbol{0}_\mcm}
\newcommandx\largestpoint[2][2=]{\left\vvvert#1\right\vvvert_{\mathsf{#2}}}  

\def\shift{B}
\def\fidi{\stackrel{fi.di.}\longrightarrow}
\newcommand\convmzero[1]
{\stackrel{\mcm_0(#1)}\longrightarrow}
\newcommandx\pseudonorm[2][2=]{\|#1\|_{\mathsf{#2}}} 
\newcommandx\dist[1][1=]{{\mathrm{d}_{\mathsf{#1}}}} 
\newcommand\zero[1]{\boldsymbol{0}_{\mathsf{#1}}} 
\newcommandx\sequence[3][2=\Zset,3=j]{\{#1_{#3},#3\in#2\}}

\newcommand\ind[1]{\mathbbm{1}{\left\{#1\right\}}}

\usepackage{cleveref}
\numberwithin{equation}{section}
\theoremstyle{plain}
\newtheorem{theorem}{Theorem}[section]

\newtheorem{lemma}[theorem]{Lemma}
\newtheorem{corollary}[theorem]{Corollary}
\newtheorem{proposition}[theorem]{Proposition}
\newtheorem{definition}[theorem]{Definition}
\theoremstyle{remark}

\newtheorem{remark}[theorem]{Remark}
\crefname{theorem}{Theorem}{Theorems}
\endlocaldefs

\begin{document}

\begin{frontmatter}

\title{Tail measure and  spectral tail process of regularly varying time series}
\runtitle{Tail measure of regularly varying time series}


\author{\fnms{Cl\'ement} \snm{Dombry}\corref{}\ead[label=e1]{clement.dombry@univ-fcomte.fr}}
\thankstext{t1}{The research of Cl\'ement Dombry is partially supported by the Bourgogne Franche-Comt\'e region (grant OPE-2017-0068).}
\address{Cl\'ement Dombry\\ Universit{\'e} Bourgogne Franche-Comt{\'e},\\ Laboratoire de Math{\'e}matiques de Besan\c{c}on,\\ UMR CNRS 6623,\\ 16 route de Gray, \\ 25030 Besan{\c c}on Cedex, France.\\ \printead{e1}}
\affiliation{Universit{\'e} Bourgogne Franche-Comt\'e}
\and
\author{\fnms{Enkelejd} \snm{Hashorva}\ead[label=e2]{enkelejd.hashorva@unil.ch}}
\thankstext{t2}{ The research of Enkelejd Hashorva is partially supported by the SNSF Grant  no. 200021-175752/1 }
\address{ Enkelejd Hashorva, \\ Universit{\'e} de Lausanne,\\ D{\'e}partement de sciences actuarielles,\\ Quartier  UNIL-Chamberonne,\\ B{\^a}timent Extranef,\\ 1015 Lausanne, Switzerland.\\\printead{e2}}
\affiliation{Universit{\'e} de Lausanne}
\and
\author{\fnms{Philippe} \snm{Soulier}\ead[label=e3]{philippe.soulier@u-paris10.fr}}
\thankstext{t3}{ The research of Philippe Soulier is partially supported by LABEX MME-DII.}
\address{Philippe Soulier,\\ Universit\'e    Paris Nanterre,\\ D{\'e}partement de Math{\'e}matique et informatique,\\ Laboratoire MODAL'X,\\ 92000 Nanterre, France.\\ \printead{e3}}
\affiliation{Universit\'e    Paris Nanterre}

\runauthor{C.~Dombry, E.~Hashorva and P.~Soulier}

\begin{abstract}
The goal of this paper is an exhaustive investigation of the link between the tail measure of a
  regularly varying time series and its spectral tail process, independently introduced in
  \cite{samorodnitsky:owada:2012} and \cite{basrak:segers:2009}. Our main result is to prove in an
  abstract framework that there is a one to one correspondance between these two objets, and given
  one of them to show that it is always possible to build a time series of which it will be the tail
  measure or the spectral tail process. For non negative time series, we recover results explicitly
  or implicitly known in the theory of max-stable processes.
\end{abstract}



\end{frontmatter}

\section{Introduction}
Regular variation is a fundamental concept for the extreme value analysis of time series. See for
instance \cite{kulik:2016} and the articles in this collection for a recent overview.  For
stationary multivariate time series, \cite{basrak:segers:2009} proved that regular variation is
equivalent to the existence of the so-called tail and spectral tail processes which capture the
entire tail behaviour of the series. An important property of the spectral tail process is the time
change formula also proved by \cite{basrak:segers:2009}. Recently, \cite{segers:zhao:meinguet:2017}
and \cite{samorodnitsky:owada:2012} introduced the tail measure of a regularly varying, but not
necessarily stationary, time series. The tail measure is a homogeneous measure on the sequence space
and it is shift-invariant for a stationary time series. This is an advantage with respect to the
tail process which is never stationary. In addition, the tail process can be recovered from the tail
measure and it appears that the time change formula is a straightforward consequence of the shift
invariance of the tail measure.

A very natural question arises: given the tail process or the spectral tail process of a time
series, is it possible to reconstruct explicitly the tail measure? Furthermore, since the tail and
spectral tail processes can be defined solely in terms of the tail measure, given a process
satisfying the time change formula, is it possible to define a tail measure and a time series of
which it is the spectral tail process? The latter question was recently solved positively by
\cite{janssen:2017} who shows that given a process satisfying the time change formula, there exists
a time series of which it is the spectral tail process.

The purpose of this paper is twofold.  In \Cref{sec:tailmeasurerepresenation} we will attempt to present a systematic theory of tail measures on a abstract complete separable metric space and their
representations, with a particular focus on shift-invariant tail measures.  This is done by means
only of measure theory and the homogeneity and shift invariance properties of a tail measure,
without any appeal to regular variation or probabilistic asymptotic arguments. We establish in
\Cref{theo:tiltshift} the stochastic representation of tail measures with a characterization of the
shift invariance. These stochastic representations have a property similiar to the time change
formula which we refer to as the tilt shift formula.  The spectral tail process associated to the
tail measure is then related to its stochastic representation and we prove that there is a
one-to-one correspondance between spectral tail processes, stochastic representations and shift
invariant tail measures in \Cref{theo:equivalent-tailmeasure-spectraltailprocess}.

In \Cref{sec:dissipative}, we discuss dissipative representations of tail measures and characterize
the existence of such representations, which are deeply related to the mixed moving average
representation of max-stable processes. We conclude this general investigation of tail measure by
introducing maximal indices which extend the candidate extremal index of \cite{basrak:segers:2009}.

In Section~\ref{sec:regvar-ts}, the abstract tail measures introduced in
\Cref{sec:tailmeasurerepresenation} are related to be the tail measure of a regularly varying time
series, in particular max-stable processes - see  \cite{dH84},  \cite{KLP2} or  \cite{KLP1}.  The main result of this section is that we show that any shift-invariant homogeneous measure $\tailmeasure$ can be obtained as the tail measure of a
regularly varying stationary times series. Our construction relies on a Poisson particle system,
similarly to the representation of max-stable sequences, and on the regular variation of Poisson
point measures on abstract metric spaces. The main theoretical tool we use is the theory of $\mcm_0$
convergence on metric spaces and its application to regular variation, following
\cite{hult:lindskog:2006}. We also make use of the theory of convergence of random measures as set
out in \cite{kallenberg:2017}. Our main result extend the above mentioned result of
\cite{janssen:2017} to our more general framework.  The properties of the proposed class of
stationary regularly varying time series are then studied and we show in particular that they admit
extremal indices which coincide with the maximal indices introduced in \Cref{sec:maxindices}.

\section{Tail measures on a metric space}
\label{sec:tailmeasurerepresenation}

\subsection{Framework}
\label{sec:framework}
The mathematical setting is the following. Let $(\setE,\mce)$ be a measurable cone, that is a
measurable space together with a multiplication by positive scalars
\[
(u,\bx)\in (0,\infty)\times \setE \mapsto u \bx\in\setE \; , 
\]
which is measurable with respect to the product $\sigma$-field $\mathcal{B}(0,\infty)\otimes\mce/\mce$
and satisfies
\[
1\bx=\bx \; , \ \  u(v(\bx)) = (uv)\bx \; , \  \ u,v>0 \; , \ \ \bx\in\setE \; .
\]
We assume that the cone admits a zero element $\bszero_\setE\in\setE$ such that
$u\zero\setE=\zero\setE$ for all $u>0$ and that it is endowed with a pseudonorm, i.e. a measurable
function $\pseudonorm{\cdot}[\setE]: \setE \mapsto [0,\infty)$ such that
$\pseudonorm{u\bx}[\setE]=u\pseudonorm{\bx}[\setE]$ for all $u>0$, $\bx\in\setE$ and
$\pseudonorm{\bx}[\setE]=0$ implies $\bx=\zero\setE$. The triangle inequality is not required.

The space $\seqspace{\setE}$ of $\setE$-valued sequences is endowed with the cylinder
$\sigma$-algebra $\mcf=\mce^{\otimes \Zset}$ and a generic sequence is denoted
$\bx=(\bx_h)_{h\in\Zset}$.  The sequence identically equal to $\bszero_E$ is denoted by
$\zeroseq\setE$.  The backshift operator $\shift$ on $\seqspace\setE$ is defined by
$(\shift\bx)_h=\bx_{h-1}$, $\bx\in\seqspace\setE$, $h\in\Zset$. Its iterates are denoted $\shift^k$,
$k\in\Zset$.

Let $H: \seqspace\setE \mapsto [0,\infty]$ be an $\mcf$-measurable function. We say that $H$ is
homogeneous of order $\alpha\in\Rset$, or shortly $\alpha$-homogeneous, if $H(u\bx)=u^\alpha H(\bx)$
for all $u>0$, $\bx\in\seqspace\setE$.

The central object in this section is the notion of tail measure defined as follows.
\begin{definition}[Tail measure]\label{def:tailmeasure}
  \label{def:tail-measure}
  A tail measure with index $\alpha>0$ is a positive measure $\tailmeasure$ on $(\seqspace\setE,\mcf)$ with the
  following properties: 
  \begin{enumerate}[(i),wide=1pt]
  \item \label{item:zerozero} $\tailmeasure(\{\zeroseq\setE\})=0$;
  \item \label{item:standardization}  $\tailmeasure(\{\pseudonorm{\bx_0}[E]>1\})=1$; 
  \item \label{item:finipourtout-h}  $\tailmeasure(\{\pseudonorm{\bx_h}[E]>1\})<\infty$ for all $h\in\Zset$; 
  \item \label{item:homogeneity} $\tailmeasure$ is $\alpha$-homogeneous, that is
    $\tailmeasure(uA)=u^{-\alpha}\tailmeasure(A)$ for all $A\in\mcf$ and $u>0$.
  \end{enumerate}
  The tail measure $\tailmeasure$ is called shift-invariant  if furthermore
\begin{enumerate}[(i),wide=1pt]
  \item[(v)] $\tailmeasure(\shift A)=\tailmeasure(A)$ for all $A\in\mcf$. 
  \end{enumerate}
\end{definition}

The following connection of tail measures on $\seqspace{[0,\infty)}$ and max-stable process is
important.
\begin{remark}
  \label{rem:max-stable1}
  A time series $\bX=(\bX_h)_{h\in\Zset}$ is called $\alpha$-Fr\'echet max-stable if
  \[
  \Big(n^{-1/\alpha}\bigvee_{i=1}^n \bX^{(i)}_h\Big)_{h\in\Zset} \stackrel{d}= (\bX_h)_{h\in\Zset}
  \]
  where $\bX^{(i)}$, $i\geq 1$ are independent copies of $\bX$. de Haan's representation theorem
  \citep{dH84} implies that any $\alpha$-Fr\'echet max-stable sequence $\bX$ can be represented as
  \begin{equation}
    \label{def:max-stable}
    (\bX_h)_{h\in\Zset}\stackrel{d}=\Big(\bigvee_{i\geq 1} \mathbf{P}^{(i)}_h\Big)_{h\in\Zset}
  \end{equation}
  where $\sum_{i\geq 1} \delta_{\mathbf{P}^{(i)}}$ is a Poisson random measure on
  $\seqspace{[0,\infty)}$ with intensity $\tailmeasure$ called the exponent measure of
  $\bX$. Provided the marginal distribution of $\bX_0$ is standard $\alpha$-Fr\'echet, the exponent
  measure $\tailmeasure$ is a tail measure in the sense of \Cref{def:tail-measure}. Conversely, for
  any tail measure $\tailmeasure$ on $\seqspace{[0,\infty)}$, \Cref{def:max-stable} defines an
  $\alpha$-Fr\'echet max-stable sequence with $\bX_0$ following a standard $\alpha$-Fr\'echet
  distribution.
\end{remark}

The following lemma is 
	useful to characterize tail measures. According to Definition~\ref{def:tail-measure}, the restriction of a tail measure $\tailmeasure$ to the set $\{\pseudonorm{\bx_h}[E]>1\}$ is finite so that the Lemma allows to deal with  finite  measures in order to characterize $\tailmeasure$.  

\begin{lemma}
  \label{lem:determination-tailmeasure}
  Any tail measure $\tailmeasure$ is $\sigma$-finite and   uniquely determined by its restrictions to the sets
  $\{\pseudonorm{\bx_h}[E]>1\}$, $h\in\Zset$.
\end{lemma}

\begin{proof}
 By property~\ref{item:zerozero} of Definition~\ref{def:tailmeasure}, the tail measure $\tailmeasure$ is supported by
 \begin{align*}
   \seqspace\setE\setminus\{\zeroseq\setE\}\,=\,\bigcup_{h\in\Zset, n\geq 1} {A}_{h,n} \; , 
 \end{align*}
 with $ A_{h,n}=\{\bx\in\seqspace\setE : \pseudonorm{\bx_h}[E]>n^{-1}\}$.  Since
 $\seqspace\setE\setminus\{\zeroseq\setE\}$ is a countable union of measurable sets, we can also
 write it as a countable union of pairwise disjoint measurable sets. For instance enumerating
 ${A}_{h,n}, h\in \Zset, n\ge 1$ as $D_i,i\ge 1$ and taking
 $\mathcal D_1=D_1, \mathcal D_i= D_i \cap (D_1 \cup \cdots \cup D_{i-1})^c, i\ge 2$,  we have that
 $\seqspace\setE\setminus\{\zeroseq\setE\}= \cup_{i\ge 1} \mathcal D_i$ with the sets
 $\mathcal D_i$, $i\geq1$, being pairwise disjoint. Since by property~\ref{item:finipourtout-h}
 and~\ref{item:homogeneity} we have that $\tailmeasure(\mathcal D_i) < \infty, i\ge 1$, then
 $\tailmeasure$ is $\sigma$-finite and completely determined by its restrictions to the sets
 $\mathcal{D}_i$, $i\geq1$, hence by its restriction to the sets $A_{h,n}$, $h\in\Zset$,
 $n\geq1$. Using further the homogeneity property~\ref{item:homogeneity}, it follows that
 $\tailmeasure$ is determined by its restriction to the sets $\{\pseudonorm{\bx_h}[E]>1\}$,
 $h\in\Zset$.
 \end{proof}


\subsection{Stochastic representation of  tail measures}

The following theorem provides a fundamental stochastic representation of  a tail measure in terms of a $\setE$-valued stochastic process $\bZ=(\bZ_h)_{h\in\Zset}$ and
characterizes shift-invariant tail measures. 

\begin{theorem}
  \label{theo:tiltshift}
  A measure $\tailmeasure$ on $(\seqspace\setE,\mcf)$ is a tail measure with index $\alpha>0$ if and
  only if there exists an $\setE$-valued stochastic process $\bZ=(\bZ_h)_{h\in\Zset}$ defined on a
  probability space $(\Omega,\mca,\pr)$ such that
  \begin{equation}
    \label{eq:cond-Z}
    \pr(\bZ=\zeroseq\setE)=0\;,\quad \esp[\pseudonorm{\bZ_0}[E]^\alpha]=1\;, \quad \esp[\pseudonorm{\bZ_h}[E]^\alpha]<\infty \mbox{  for all $h\in\Zset$}\;,
  \end{equation}
  and
  \begin{align}
    \tailmeasure(A) =  \int_0^\infty \pr(r\bZ\in A) \alpha r^{-\alpha-1} \rmd r\;,\quad A\in\mcf \; . \label{eq:polar-nu-Z}
  \end{align}
  Moreover, $\tailmeasure$ is shift-invariant if and only if, for all non negative measurable
  $\alpha$-homogeneous functions $H$ and $h\in\Zset$,
  \begin{align}
    \label{eq:alpha-shift-invariance}
    \esp[H(\shift^h\bZ)] =\esp[ H(\bZ)] \; . 
  \end{align}
\end{theorem}
Note that in Equation~(\ref{eq:polar-nu-Z}), both terms may be equal to $+\infty$, for instance if $A=\{\pseudonorm{\bx}[E]>0\}$. This raises however no difficulty since the results from measure theory we use (e.g. Fubini-Tonneli theorem)  hold true for any non-negative functions and $\sigma$-finite measures, regardless the integrals are finite or not.

We call the identity~(\ref{eq:alpha-shift-invariance}) the tilt shift formula, abbreviated TSF. It
 characterizes the shift-invariance of the measure $\tailmeasure$ defined by
\eqref{eq:polar-nu-Z} which does not depend on the choice of a norm.  It looks very much like
stationarity of the process $\bZ$, but let us emphasize that \eqref{eq:alpha-shift-invariance} is restricted to $\alpha$-homogeneous test
functions so it is much weaker than stationarity. Of course, if $\bZ$ is stationary then it
satisfies \eqref{eq:alpha-shift-invariance}.  The tilt shift formula is equivalent to each of the
following equivalent conditions which will also be referred to indifferently as the TSF:
\begin{enumerate}[(i)]
\item \label{item:0-homogeneous} for all non negative measurable $0$-homogeneous functions
  $H_0:\seqspace\setE\to[0,\infty]$ and $h\in\Zset$, 
  \begin{align}
    \label{eq:tiltshift}
    \esp[\pseudonorm{\bZ_0}[E]^\alpha H_0(\shift^h\bZ)] =\esp[\pseudonorm{\bZ_h}[E]^\alpha H_0(\bZ)] \; ; 
  \end{align}
\item \label{item:non-homogeneous} for all non negative measurable functions
  $K:\seqspace\setE\to\Rset$ and $h\in\Zset$,
  \begin{align}
    \label{eq:tiltshift-general}
    \esp[\pseudonorm{\bZ_0}[E]^\alpha K(\pseudonorm{\bZ_0}[E]^{-1}\shift^h\bZ)] =\esp[\pseudonorm{\bZ_h}[E]^\alpha K(\pseudonorm{\bZ_h}[E]^{-1}\bZ)] \; .
  \end{align}
\end{enumerate}
Indeed, (\ref{eq:tiltshift-general}) obviously implies (\ref{eq:alpha-shift-invariance})
and~(\ref{eq:tiltshift}), (\ref{eq:tiltshift}) is obtained by applying
(\ref{eq:alpha-shift-invariance}) to the $\alpha$-homogeneous function
$H(\bx) = \pseudonorm{\bx_0}[E]^\alpha H_0(\bx)$ and (\ref{eq:tiltshift-general}) is obtained by
applying~(\ref{eq:alpha-shift-invariance}) to the $\alpha$-homogeneous function
$H_0(\bx)=\pseudonorm{\bx_0}[E]^\alpha K(\pseudonorm{\bx_0}[E]^{-1}\shift^h\bx)$ defined to be 0 if $\pseudonorm{\bx_0}[E]=0$.

\begin{remark}
  \label{rem:max-stable2}
  In the case $\setE=\seqspace{[0,\infty)}$, if $\tailmeasure$ has representation
  \Cref{eq:polar-nu-Z}, then the max-stable process $\bX$ with exponent measure $\tailmeasure$
  defined by \eqref{def:max-stable} can be represented as
\[
(\bX_h)_{h\in\Zset}\stackrel{d}=\Big(\bigvee_{i\geq 1} U_i \bZ^{(i)}_h\Big)_{h\in\Zset}
\]
where $\sum_{i\geq 1} \delta_{U_i}$ is a Poisson random measure on $(0,\infty)$ with intensity
$\alpha u^{-\alpha-1}\rmd{u}$ and, independently, $\bZ^{(i)}$, $i\geq 1$, are independent copies of
$\bZ$. We note in passing that TSF for Brown-Resnick max-stable processes first appears in \cite[Lemma 5.2]{DiM}, see also \cite[Theorem 6.9]{Htilt} for general max-stable processes.
\end{remark}

\begin{proof}[Proof of \Cref{theo:tiltshift}]
  It is easily checked that the measure $\tailmeasure$ defined by \eqref{eq:polar-nu-Z} is a tail
  measure. The condition $\tailmeasure(\{\zeroseq\setE\})= 0$ follows from
  $\pr(\bZ=\zeroseq\setE)=0$. A direct computation yields
\[
\tailmeasure(\{\pseudonorm{\bx_h}[E]>1\})=\esp\int_0^\infty \ind{r\pseudonorm{\bZ_h}[E]>1} \,\alpha r^{-\alpha-1}\rmd r =\esp [\pseudonorm{\bZ_h }[E]^\alpha] \; , 
\]
whence we deduce 
\[
\tailmeasure\{\pseudonorm{\bx_0}[E]>1\}=\esp[\pseudonorm{\bZ_0}[E]^\alpha]=1 \; , \ \ \tailmeasure(\{\pseudonorm{\bx_h}[E]>1\})=\esp[\pseudonorm{\bZ_h}[E]^\alpha]<\infty \; .
\]
Homogeneity of order $\alpha$ follows from the simple change of variable $r'=u^{-1}r$: for all $u>0$ and $A\in\mcf$, we have
\begin{align*}
  \tailmeasure(uA)
  & =\int_0^\infty \pr(r\bZ\in uA) \,\alpha r^{-\alpha-1}\rmd r = \int_0^\infty \pr(u^{-1}r\bZ\in uA) \,\alpha r^{-\alpha-1}\rmd r \\
  & = u^{-\alpha}\int_0^\infty \pr(r\bZ\in uA) \,\alpha r^{-\alpha-1}\rmd r =u^{-\alpha}\tailmeasure(A).
\end{align*}
Conversely, let $\tailmeasure$ be a tail measure and let us prove the existence of a representation
\eqref{eq:polar-nu-Z}.  Let us first prove that there exists at least one measurable functional
$\tau:\seqspace\setE\to[0,\infty)$ having the following properties:
\begin{enumerate}[(i)]
\item \label{item:zero}  $\tau(\bx) = 0$ if and only if $\bx=\zeroseq\setE$;
\item \label{item:homogeneous}  $\tau$ is 1-homogeneous;
\item \label{item:un}  $\tailmeasure(\{\tau(\bx)>1\})=1$. 
\end{enumerate}
Define $p_h= \tailmeasure(\{\pseudonorm{\bx_h}[E]>1\})$ for $h\in\Zset$ and let
$q\in(0,\infty)^\Zset$ be a positive sequence such that $\sum_{h\in\Zset} p_hq_h^{\alpha} <
\infty$. Consider the map $\tau:\seqspace\setE\to [0,\infty]$ defined by
\begin{align*}
  \tau(\bx) = \sup_{h\in\Zset} q_h\pseudonorm{\bx_h}[E] \; . 
\end{align*}
Then $\tau$ is 1-homogeneous and since $\pseudonorm{\bx_h}[E]=0$ if and only if $\bx_h= \zero{\setE}$ for all
$h\in\Zset$, we have $\tau(\bx)=0$ if and only if $\bx=\zeroseq\setE$. By the homogeneity of
$\tailmeasure$, we have
\begin{align*}
  \tailmeasure(\{\tau(\bx)>1\}) 
  & \leq \sum_{h\in\Zset} \tailmeasure(\{q_h\pseudonorm{\bx_h}[E]>1\}) 
    = \sum_{h\in\Zset} p_h q_h^{\alpha} < \infty \; , \\
  \tailmeasure(\{\tau(\bx)>1\}) & \geq  q_0^{\alpha} \tailmeasure(\{\pseudonorm{\bx_0}[E]>1\}) =q_0^{\alpha} >0 \;,
\end{align*}
whence $\tailmeasure(\{\tau(\bx)>1\})\in (0,\infty)$. Therefore, by multiplying the sequence $q$ by a suitable
normalizing constant, we can impose that  $\tailmeasure(\{\tau(\bx)>1\})=1$.

Let now $\tau$ be an arbitrary measurable map having the properties \ref{item:zero},
\ref{item:homogeneous} and \ref{item:un} and define the ``unit sphere'' $S_\tau=\{\tau(\bx)=1\}$ and
the polar coordinate mapping
\[
\begin{array}{cccc}
  \coormap 
  & :     \seqspace\setE\setminus\{\zeroseq\setE\} & \rightarrow & (0,\infty)\times S_\tau \; , \\ 
  &    \bx &\mapsto & (\tau(\bx), \bx/\tau(\bx)) \; . 
\end{array}
\]
Define the probability measure $\sigma$ on  $S_\tau$ by
\[
\sigma(A)=\tailmeasure(\{\tau(\bx)>1,\ \bx/\tau(\bx)\in A\})\;,\quad A\in\mcf \; ,
\]  
and the measure $\nu_\alpha$ on $(0,\infty)$ with density $\alpha x^{-\alpha-1}$ with respect to
Lebesbue measure.  Since $\coormap$ is one-to-one and $\tau$ is homogeneous, we obtain the polar
representation of $\tailmeasure$, that is $\tailmeasure\circ \coormap^{-1} = \nu_\alpha\otimes \sigma$
or explicitly, for all $A\in\mcf$,
\begin{equation}
  \label{eq:nu-sigma}
  \tailmeasure(A) = \int_0^\infty\int_{\seqspace\setE}\ind {r\bx\in A}\,\sigma(\rmd\bx) \alpha r^{-\alpha-1} \rmd r \; .
\end{equation}
Indeed, starting from the right hand side of \eqref{eq:nu-sigma}, we compute 
\begin{align*}
  \int_0^\infty\int_{\seqspace\setE} 
  & \ind {r\bx\in A}\,\sigma(\rmd\bx) \alpha r^{-\alpha-1} \rmd r \\
  & = \int_0^\infty \int_{\seqspace\setE}\ind{\tau(\bx)>1,\ r\bx/\tau(\bx)\in A}\,\tailmeasure(\rmd\bx) \alpha r^{-\alpha-1} \rmd r\\
  & = \int_0^\infty \int_{\seqspace\setE}\tau(\bx)^{-\alpha}\ind{\tau(\bx)>1,\ r\bx\in A}\,\tailmeasure(\rmd\bx) \alpha r^{-\alpha-1} \rmd r\\
  & = \int_0^\infty \int_{\seqspace\setE} \tau(\bx)^{-\alpha}\ind{r\tau(\bx)>1,\ \bx\in A}\,\tailmeasure(\rmd\bx) \alpha r^{-\alpha-1} \rmd r\\
  & = \int_{\seqspace\setE} \ind{\bx\in A}\,\tailmeasure(\rmd\bx)=\tailmeasure(A) \; .
\end{align*}
We use throughout these lines that
$\tailmeasure(\{\tau(\bx)=0\})=\tailmeasure(\{\zeroseq\setE\})=0$.  The successive equalities rely
on the definition of $\sigma$, the changes of variable $r'= r/\tau(\bx)$ and $\bx'= \bx/r$, the
homogeneity of $\tailmeasure$ and $\tau$ and finally the fact that
$\int_0^\infty \ind{r>z}\alpha r^{-\alpha-1} \rmd r =z^{-\alpha}$ with $z=1/\tau(\bx)$.

Consider now a probability space $(\Omega,\mca,\pr)$ on which we can define an
$\seqspace\setE$-valued random element~$\bZ$ with distribution $\sigma$. Then (\ref{eq:nu-sigma}) is
exaclty the stochastic representation~\eqref{eq:polar-nu-Z}. The conditions in \eqref{eq:cond-Z} are
a consequence of Definition~\ref{def:tail-measure} together with \eqref{eq:polar-nu-Z}:
$\tailmeasure(\{\zeroseq\setE\})=\pr(\bZ=\zeroseq\setE)=0$ and
$\tailmeasure(\{\pseudonorm{\bx_h}[E]>1\})=\esp[\pseudonorm{\bZ_h}[E]^\alpha]$ is finite for all $h\in\Zset$ and equal to
$1$ for $h=0$.

Finally, assume that  $\tailmeasure$ is shift-invariant and let $H_0:\seqspace{\setE}\to[0,\infty]$ be a
$0$-homogeneous measurable function. Using the stochastic representation \eqref{eq:polar-nu-Z} and
Fubini-Tonelli's theorem for all $h\in\Zset$ we obtain
\begin{align*}
  \esp[\pseudonorm{\bZ_0}[E]^\alpha H_0(\shift^h\bZ)] 
  & = \esp\Bigl[ H_0(\shift^h\bZ)\int_0^\infty\ind{r\pseudonorm{\bZ_0}[E]>1} \alpha r^{-\alpha-1} \rmd r  \Bigr] \\
  & = \int_{\seqspace\setE} H_0(\shift^h\bx) \ind{\pseudonorm{\bx_0}[E]>1} \tailmeasure(\rmd\bx) \\ 
  &  = \int_{\seqspace\setE} H_0(\bx) \ind{\pseudonorm{\bx_h}[E]>1} \tailmeasure(\rmd\bx) \\ 
  & = \esp\Bigl[ H_0(\bZ)\int_0^\infty\ind{r\pseudonorm{\bZ_h}[E]>1} \alpha r^{-\alpha-1} \rmd r  \Bigr]\\ 
 &   = \esp[ H_0(\bZ)\pseudonorm{\bZ_h}[E]^\alpha] \; . 
\end{align*}
The third equality uses the shift invariance  of $\tailmeasure$ and this proves that 
\eqref{eq:tiltshift} holds.

Conversely we prove that the tilt shift formula \eqref{eq:tiltshift} implies the shift invariance of
$\tailmeasure$. For this purpose, we note that for all $h\in\Zset$ and $A\in\mcf$,
  \begin{align}
    \tailmeasure(A\cap\{\pseudonorm{\bx_h}[E]>1\})
    & = \int_0^\infty \esp[\ind{r\bZ\in A\;,\ r\pseudonorm{\bZ_h}[E]>1 }] \alpha r^{-\alpha-1} \rmd r  \nonumber\\
    & = \int_1^\infty \esp[\pseudonorm{\bZ_h}[E]^\alpha \ind{r\bZ/\pseudonorm{\bZ_h}[E]\in A}] \alpha r^{-\alpha-1} \rmd r  \nonumber\\
    & = \int_1^\infty \esp[\pseudonorm{\bZ_0}[E]^\alpha \ind{r\shift^h \bZ/\pseudonorm{\bZ_0}[E]\in A}] \alpha r^{-\alpha-1} \rmd r \; . \label{eq:mcr} 
\end{align}
We used successively the stochastic representation \eqref{eq:polar-nu-Z}, the change of variable
$r'=r\pseudonorm{\bZ_h}[E]$ (where $\pseudonorm{\bZ_h}[E]$ is almost surely finite as a consequence of \eqref{eq:cond-Z})
and the tilt-shift formula \eqref{eq:tiltshift}. Similarly, for $k\in\Zset$,
\begin{align*}
  (\tailmeasure\circ \shift^{-k})(A\cap\{\pseudonorm{\bx_h}[E]>1\})
  & = \int_0^\infty \esp[\ind{r\shift^{k}\bZ\in A\;,\ r\pseudonorm{\bZ_{h-k}}[E]>1}] \alpha r^{-\alpha-1} \rmd r   \\
  & = \int_1^\infty \esp[\pseudonorm{\bZ_{h-k}}[E]^\alpha \ind{r\shift^{k}\bZ/\pseudonorm{\bZ_{h-k}}[E]\in A}]\alpha r^{-\alpha-1} \rmd r  \\
  & = \int_1^\infty \esp[\pseudonorm{\bZ_0}[E]^\alpha \ind{r\shift^h\bZ/\pseudonorm{\bZ_0}[E]\in A}] \alpha r^{-\alpha-1} \rmd r \; .
\end{align*}
This proves that $\tailmeasure=\tailmeasure\circ\shift^{-k}$ on the set $\{\pseudonorm{\bx_h}[E]>1\}$. Since
this holds for all $h\in\Zset$, Lemma~\ref{lem:determination-tailmeasure} implies
$\tailmeasure=\tailmeasure\circ\shift^{-k}$, whence $\tailmeasure$ is shift-invariant.
\end{proof}

\subsection{The spectral tail process and the time change formula}
\label{sec:stp-tcf}
The following notion of tail process and spectral tail process plays an important role in the theory
of regularly varying time series, see \cite{basrak:segers:2009}. We define here these objects in
terms of the tail measure only. The link between these two approaches will be made in \Cref{sec:regvar-ts-construction} and was already pointed by \cite{samorodnitsky:owada:2012}, section 4.

\begin{definition}[Local tail process] 
  \label{def:local-tail-process}
  Let $\tailmeasure$ be a tail measure on $\seqspace\setE$ and assume that $h\in\Zset$ is such that
  $p_h=\tailmeasure(\{\pseudonorm{\bx_h}[E]>1\})>0$. The local tail process of $\tailmeasure$ at lag $h$ is the
  process $\bY^{(h)}$ with distribution
  \begin{align*}
    \pr(\bY^{(h)}\in A) =  \frac{1}{p_h} \tailmeasure(\{\pseudonorm{\bx_h}[E]>1,\bx \in A\})\; ,     \quad A\in\mcf\;.
  \end{align*}
  The process $\bTheta^{(h)}=\bY^{(h)}/\pseudonorm{\bY_h^{(h)}}[E]$ is called the local spectral tail process at
  lag $h$.

  For $h=0$, we write simply $\bY=\bY^{(0)}$ and $\bTheta=\bTheta^{(0)}$, called the tail process
  and the spectral tail process associated to $\tailmeasure$.
\end{definition}

\begin{proposition}
  \label{prop:localtailprocess}
  Let $\tailmeasure$ be a tail measure with stochastic representation~\eqref{eq:polar-nu-Z}. Then
  $p_h=\tailmeasure(\{\pseudonorm{\bx_h}[E]>1\})=\esp[\pseudonorm{\bZ_h}[E]^\alpha]<\infty$. If $p_h>0$, then
  $\pseudonorm{\bY_h^{(h)}}[E]$ and $\bTheta^{(h)}=\bY^{(h)}/\pseudonorm{\bY_h^{(h)}}[E]$ are independent, $\pseudonorm{\bY_h^{(h)}}[E]$
  has an $\alpha$-Pareto distribution, that is
  \[
  \pr(\pseudonorm{\bY_h^{(h)}}[E]>u)=u^{-\alpha}\;,\quad u>1\;,
  \]
  and the distribution of $\bTheta^{(h)}$ is given by
   \begin{align}
    \pr(\bTheta^{(h)}\in A) = p_h^{-1} \esp[\pseudonorm{\bZ_h}[E]^\alpha \ind{\bZ/\pseudonorm{\bZ_h}[E]\in A}] \;,\quad A\in\mcf \; .\label{eq:lp}
  \end{align}

\end{proposition}

\begin{proof}
  By definition of the local tail process and using the stochastic representation
  \eqref{eq:polar-nu-Z}, we have for all measurable $H:\seqspace\setE\to [0,\infty]$
  \begin{align*}
    p_h\esp[H(\bY^{(h)})] 
    & = \int_{\seqspace\setE} H(\bx) \ind{\pseudonorm{\bx_h}[E]>1} \tailmeasure(\rmd\bx) \\
    & = \int_0^\infty \esp[H(r\bZ) \ind{r\pseudonorm{\bZ_h}[E]>1}] \alpha r^{-\alpha-1} \rmd r \\
    & = \int_0^\infty \esp[H(r\bZ) \ind{r\pseudonorm{\bZ_h}[E]>1,\ 0< \norm{\bZ_h}< \infty}] \alpha r^{-\alpha-1} \rmd r \\
    & = \int_1^\infty \esp\left[\pseudonorm{\bZ_h}[E]^\alpha H\left(r\bZ/\pseudonorm{\bZ_h}[E]\right)  \right] \alpha r^{-\alpha-1} \rmd r \; .
  \end{align*}
  The last equality relies on the change of variable $r'=r\pseudonorm{\bZ_h}[E]$.  Applying this identity with
  the $0$-homogeneous function $H_0(\bx)= \ind{\bx/\pseudonorm{\bx_h}[E]\in A,\ \pseudonorm{\bx_h}[E]>0}$ yields
  \begin{align*}
    \pr ( \bTheta^{(h)}\in A ) 
    & =  \esp[H_0(\bY^{(h)})] =  \int_1^\infty \esp\left[\pseudonorm{\bZ_h}[E]^\alpha H_0(\bZ/\pseudonorm{\bZ_h}[E]) \right] \alpha r^{-\alpha-1} \rmd r \\
    & = p_h^{-1} \esp\left[\pseudonorm{\bZ_h}[E]^\alpha \ind{\bZ/\pseudonorm{\bZ_h}[E]\in A } \right] \; ,
\end{align*}
proving Equation~\eqref{eq:lp}.
\end{proof}

\begin{corollary}
  \label{cor:TCF}
  A tail measure $\tailmeasure$ is shift-invariant if and only if $p_h=1$ and
  $\bTheta^{(h)}\stackrel{d}=\shift^{h} \bTheta$ for all $h\in\Zset$.  Then the spectral tail
  process $\bTheta$ characterizes the tail measure $\tailmeasure$ and satisfies
  \begin{align}
   \esp[H_0(\shift^h\bTheta)] = \esp[\pseudonorm{\bTheta_h}[E]^\alpha H_0(\bTheta)] \; , \ \  h\in\Zset \; ,
    \label{eq:time-change-formula-0}    
  \end{align}
  for all $0$-homogeneous  measurable $H_0:\seqspace\setE\to [0,\infty)$ vanishing on $\{\pseudonorm{\bx_0}[E]=0\}$.
\end{corollary}
We call Equation~\eqref{eq:time-change-formula-0} the time change formula, abbreviated TCF. It first
appeared in \cite{basrak:segers:2009} in the context of stationary regularly varying time
series. While the original proof was based on limiting arguments, we propose here a direct proof
based on shift invariance of the tail measure, which was already noticed in  \cite{samorodnitsky:owada:2012}.  In view of Proposition~\ref{prop:localtailprocess},
the TCF is in fact a direct consequence of the TSF (see the proof below).  The condition that $H_0$
vanishes on $\{\pseudonorm{\bx_0}[E]=0\}$ is important. To stress this, the TCF can be formulated in the
equivalent form: for all $0$-homogeneous measurable $H_0:\seqspace\setE\to [0,\infty)$,
  \begin{align}
    \esp[H_0(\shift^h\bTheta) \ind{\pseudonorm{\bTheta_{-h}}[E]>0}] = \esp[\pseudonorm{\bTheta_h}[E]^\alpha H_0(\bTheta)] \; ,\ \  h\in\Zset \; .
    \label{eq:time-change-formula-1}    
  \end{align}
To see this, simply apply \eqref{eq:time-change-formula-0} to the function $\bx\mapsto H_0(\bx)\ind{\pseudonorm{\bx_0}[E]>0}$.

\begin{proof}[Proof of Corollary~\ref{cor:TCF}]
  If $\tailmeasure$ is shift-invariant, then the tilt shift formula~(\ref{eq:tiltshift}) together
  with $\esp[\pseudonorm{\bZ_0}[E]^\alpha]=1$ implies $p_h=1$ for all
  $h\in\Zset$. Equations~\eqref{eq:tiltshift} and~\eqref{eq:lp} together imply, for all
  $h\in \Zset$, $A\in\mcf$,
  \begin{align*}
    \pr(\bTheta^{(h)}\in A) 
    & = \esp[\pseudonorm{\bZ_h}[E]^\alpha \ind{\bZ/\pseudonorm{\bZ_h}[E]\in A}] \\
    & = \esp[\pseudonorm{\bZ_0}[E]^\alpha \ind{B^h\bZ/\pseudonorm{\bZ_0}[E]\in A}] =     \pr(B^h \bTheta\in A) \; , 
  \end{align*}
  whence $\bTheta^{(h)}\stackrel{d}=B^h \bTheta$. Conversely, if $p_h=1$ and
  $\bTheta^{(h)}\stackrel{d}=B^h \bTheta$ for all $h\in \Zset$, then we have for all 0-homogeneous
  function $H_0$,
  \begin{align*}
    \esp\left[\pseudonorm{\bZ_0}[E]^\alpha  H_0(\shift^h\bZ) \right] =   \esp[H_0(\shift^{h}\bTheta)]  
    =  \esp[H_0(\bTheta^{(h)})]=  \esp\left[\pseudonorm{\bZ_h}[E]^\alpha H_0(\bZ)  \right] \; , 
  \end{align*}
  hence the TSF is satisfied and $\tailmeasure$ is shift-invariant by \Cref{theo:tiltshift}.

  If $\tailmeasure$ is shift-invariant, then Equation~\eqref{eq:mcr} can be rewritten as
  \[
  \tailmeasure(A\cap\{\pseudonorm{\bx_h}[E]>1\})= \int_1^\infty \esp[ \ind{r\shift^h \bTheta \in A}] \alpha
  r^{-\alpha-1} \rmd r \; , 
  \]
  for all $h\in\Zset$ and $A\in\mcf$. In view of Lemma~\ref{lem:determination-tailmeasure}, we
  deduce that $\bTheta$ characterizes the shift-invariant tail measure $\tailmeasure$.  Furthermore,
  we have for all $0$-homogeneous function $H_0$
\begin{align*}
  \esp[\pseudonorm{\bTheta_h}[E]^\alpha H_0(\bTheta) ] 
  & = \esp \left[ \pseudonorm{\bZ_0}[E]^\alpha \frac{\pseudonorm{\bZ_h}[E]^\alpha}{\pseudonorm{\bZ_0}[E]^\alpha}  H_0(\bZ/\pseudonorm{\bZ_0}[E]) \right] 
    = \esp \left[  \pseudonorm{\bZ_h}[E]^\alpha  H_0(\bZ) \ind{\pseudonorm{\bZ_0}[E]> 0}\right]  \\
  &  =\esp \left[  \pseudonorm{\bZ_h}[E]^\alpha  H_0(\bZ) \right]=\esp[H_0(\shift^h\bTheta)] \; , 
\end{align*}
where the second line of equalities is valid provided that $H_0$ vanishes on $\{ \pseudonorm{\bx_0}[E]=0\}$.
This shows that if $\tailmeasure$ is shift-invariant, the spectral tail process satisfies the
TCF~\eqref{eq:time-change-formula-0}.
\end{proof}

We have introduced the spectral tail process $\bTheta$ associated to a tail measure $\tailmeasure$.
In the shift-invariant case, it satisfies $\pr(\pseudonorm{\bTheta_0}[E]=1)=1$ and the TCF
\eqref{eq:time-change-formula-0} and it also characterizes~$\tailmeasure$. A natural question then
arises: if $\bTheta$ satisfies $\pr(\pseudonorm{\bTheta_0}[E]=1)=1$ and the TCF, can it be obtained as the
spectral tail process of some shift-invariant tail measure $\tailmeasure$? In the multivariate
setting $\setE=\mathbb{R}^d$, this question was addressed recently by \cite{janssen:2017} in
connection with the theory of max-stable processes. The next theorem still provides a positive
answer in the more general framework. Our proofs are different and work directly on the level of the
tail measure (not on the level of a stationary regularly varying time series, see
\Cref{theo:nu-to-X} below).
 
\begin{theorem}
  \label{theo:equivalent-tailmeasure-spectraltailprocess}
  The mapping which to a tail measure associates its spectral tail process is a one-to-one
  correspondence between the class of shift-invariant tail measures and the class of processes
  $\bTheta$ satisfying $\pr(\pseudonorm{\bTheta_0}[E]=1)=1$ and the TCF~(\ref{eq:time-change-formula-0}).
\end{theorem}
\begin{proof}
  Starting from a process $\bTheta$ satisfying $\pr(\pseudonorm{\bTheta_0}[E]=1)=1$ and the
  TCF~(\ref{eq:time-change-formula-0}), we need to construct a shift-invariant tail measure
  $\tailmeasure$ with spectral tail process $\bTheta$. For $q\in[0,\infty)^\Zset$ and
  $\bx\in\seqspace\setE$, we define
  \begin{align*}
    \norm{\bx}_{q,\alpha} = \left(\sum_{j\in\Zset} q_j\pseudonorm{\bx_j}[E]^\alpha\right)^{1/\alpha} \; .
  \end{align*}
  We can always choose the sequence $q$ such that
  \begin{align}
    \pr(0<\norm{\shift^k\bTheta}_{q,\alpha}<\infty) = 1 \quad \mbox{for all $k\in\Zset$}\; .
  \end{align}
  It suffices to choose $q$ such that $q_k>0$ for all $k\in\Zset$ and $\sum_{k\in\Zset} q_k=1$. Then
  $\pr(\norm{\shift^k\bTheta}_{q,\alpha}>0) = 1$ since $q_k>0$ and $\pseudonorm{\bTheta_0}[E]=1$ almost
  surely. Moreover, applying the time change formula (\ref{eq:time-change-formula-1}) with
  $H\equiv1$ yields $\esp[\pseudonorm{\bTheta_k}[E]^\alpha ]\leq1$ for all $k\in\Zset$, so that
  \begin{align*}
    \esp[\norm{\shift^k\bTheta}_{q,\alpha}^\alpha] = \sum_{j\in\Zset}q_k \esp[\pseudonorm{\shift^k\bTheta_j}[E]^\alpha]   \leq 
    \sum_{k\in\Zset} q_k =1 \; . 
  \end{align*}
  Define $\tbZ^{(k)} = \shift^k\bTheta/\norm{\shift^k\bTheta}_{q,\alpha}$, $k\in\Zset$ and the positive measure
  $\tailmeasure_q$ on $\seqspace\setE$ by
  \begin{align}
    \label{eq:def-nuq}
    \tailmeasure_q(A) & = \sum_{k\in\Zset} q_k \int_0^\infty \mathbb{P}(r\tbZ^{(k)}\in A) \alpha r^{-\alpha-1} \rmd r   
  \end{align}
for all $A\in\mathcal{F}$.   Then $\tailmeasure_q$ is obviously   $\alpha$-homogeneous and $\tailmeasure_q(\{\zeroseq\setE\})=0$ and we have furthermore, for all measurable function $H:\seqspace\setE\to[0,\infty)$, 
\begin{align*}
  \int_{\seqspace\setE} H(\bx) 
  & \ind{\pseudonorm{\bx_0}[E]>1} \tailmeasure_q(\rmd\bx)  \\
  & = \sum_{k\in\Zset} q_k \int_0^\infty \esp\Bigl[H(r\tbZ^{(k)})\ind{ r\pseudonorm{\tbZ^{(k)}_{0}}[E]>1 } \Bigr] \alpha r^{-\alpha-1} \rmd r  \\
  & = \sum_{k\in\Zset} q_k \int_0^\infty \esp\left[
    H\left(\frac{r\shift^k\bTheta}{\norm{\shift^k\bTheta}_{q,\alpha}}\right)
    \ind{r\pseudonorm{\bTheta_{-k}}[E] >   \norm{\shift^k\bTheta}_{q,\alpha} } \right] \alpha r^{-\alpha-1} \rmd r  \\
  & = \sum_{k\in\Zset} q_k \int_1^\infty \esp \left[ H\left(\frac{ r\shift^k\bTheta}{\pseudonorm{\bTheta_{-k}}[E]}\right)
    \frac{\pseudonorm{\bTheta_{-k}}[E]^\alpha}{\norm{\shift^k\bTheta}_{q,\alpha}^\alpha} \right]  \alpha r^{-\alpha-1} \rmd r \\
  & = \sum_{k\in\Zset} q_k \int_1^\infty \esp \left[\pseudonorm{\bTheta_{-k}}[E]^\alpha H\left(\frac{r\shift^k\bTheta}{\norm{\bTheta_{-k}}}\right)
    \frac{\norm{\bTheta_0}^\alpha}{\norm{\shift^k\bTheta}_{q,\alpha}^\alpha} \right] \alpha r^{-\alpha-1} \rmd r \;.  
\end{align*}
In these lines, we used successively the definition \eqref{eq:def-nuq}, the definition of $\tbZ$,
the change of variable $r'=\norm{\shift^k\bTheta}_{q,\alpha}\pseudonorm{\bTheta_{-k}}[E]^{-1}r$
(note that the event $\{\pseudonorm{\bTheta_{-k}}[E]=0\}$ has no contribution to the expectations)
and finally the fact that $\pr(\norm{\bTheta_0}=1)=1$.  The time change formula now entails
\begin{align}
  \int_{\seqspace\setE} H(\bx) \ind{\pseudonorm{\bx_0}[E]>1} \tailmeasure_q(\rmd\bx) 
  & = \sum_{k\in\Zset} q_k \int_1^\infty \esp \left[ H\left(\frac{r\bTheta}{\pseudonorm{\bTheta_0}[E]}\right) 
    \frac{\pseudonorm{\bTheta_{k}}[E]^\alpha}{\norm{\bTheta}_{q,\alpha}^\alpha}    \right] \alpha r^{-\alpha-1} \rmd r \nonumber \\
  & = \int_1^\infty \esp[H(r\bTheta)]  \alpha r^{-\alpha-1} \rmd r \;. \label{eq:proof-identity}
\end{align}
Applying this identity to the 0-homogeneous function $\bx\to H(\pseudonorm{\bx_0}[E]^{-1}\bx)\ind{\pseudonorm{\bx_0}[E]>0}$
proves that $\tailmeasure_q$ has spectral tail process $\bTheta$. It is easily obtained along the
same lines, that for all~$h\in\Zset$, 
\begin{align*}
  \int_{\seqspace\setE} H(\bx) \ind{\pseudonorm{\bx_h}[E]>1} \tailmeasure_q(\rmd\bx) 
  =  \int_1^\infty \esp[H(r\shift^h\bTheta)]  \alpha r^{-\alpha-1} \rmd r \; . 
\end{align*}
The right hand side does not depend on $q$ and taking $H\equiv 1$ yields
$\tailmeasure_q(\{\pseudonorm{\bx_h}[E]>1\})=1$, $h\in \Zset$. Therefore the $\tailmeasure_q$'s are tail
measures that coincide on the sets $\{\pseudonorm{\bx_h}[E]>1\}$, $h\in \Zset$. By
\Cref{lem:determination-tailmeasure} they are all equal and hence $\tailmeasure_q$ does not depend
$q$. This entails that~$\tailmeasure_q$ is shift-invariant since it is readily checked that
$\tailmeasure_q\circ \shift^{-h}=\tailmeasure_{\shift^h q}$ whence
$\tailmeasure_q\circ \shift^{-h}=\tailmeasure_{q}$.
\end{proof}

\begin{remark}
  \label{rem:dissipative-1}
    In two particular cases, a simpler construction of the tail measure corresponding to a given
  spectral tail process is available. 
    \begin{enumerate}[$\bullet$,wide=0pt]
  \item If $ \pr(\pseudonorm{\bTheta_h}[E]>0)=1$ for all $h\in\Zset$, then the sequence $q$ can be chosen
    as $q=\delta_0$ and we obtain
    \begin{align*}
      \tailmeasure(A) = \int_0^\infty \pr(r \bTheta\in A)\alpha r^{-\alpha-1} \rmd r \; .
    \end{align*}
    This provides a stochastic representation \eqref{eq:polar-nu-Z} of $\tailmeasure$ with
    $\bZ=\bTheta$ such that $\pseudonorm{\bZ_0}[E]=\pseudonorm{\bTheta_0}[E]=1$ almost surely.

  \item If $\pr(\sum_{h\in\Zset} \pseudonorm{\bTheta_h}[E]^\alpha < \infty)=1$, then we can choose $q\equiv 1$
    which yields
    \begin{align}
      \label{eq:dissipative-1}
      \tailmeasure(A) = \sum_{h\in\Zset} \int_0^\infty \pr(r\shift^h \tbZ \in A)\alpha r^{-\alpha-1} \rmd r \; ,    
    \end{align}
    with $\tbZ=\bTheta/(\sum_{k\in\Zset}\pseudonorm{\bTheta_k}[E]^\alpha)^{1/\alpha}$.  This representation is
    related to the mixed moving maximum representation of max-stable process see e.g. \cite{DK2016}
    and \Cref{sec:dissipative}.
\end{enumerate}
\end{remark}

We will later need the following lemma on the support of a tail measure. We say that a set
$C\in\mcf$ is a cone if $\bx\in C$ implies $u\bx\in C$ for all $u>0$.
\begin{lemma}
  \label{lem:support-cone}
  Let $\tailmeasure$ be a tail measure which admits the stochastic
  representation~(\ref{eq:polar-nu-Z}).  Let $C$ be a cone. Then
  $\tailmeasure(C)=0\Leftrightarrow\pr(\bZ\in C)=0$. If $\tailmeasure$ and  $C$
  are  shift-invariant, then   $\tailmeasure(C)=0\Leftrightarrow\pr(\bTheta\in C)=0$. 
\end{lemma}
\begin{proof}
  If $C$ is a cone, then Equation~\eqref{eq:polar-nu-Z} yields
  $\tailmeasure(C) = \pr(\bZ\in C) \times \infty$, which proves the first statement.

  If $\tailmeasure$ is shift-invariant with spectral tail process $\bTheta$ and $C$ is a shift
  invariant cone, the representation \eqref{eq:def-nuq} yields
  \begin{align*}
    \tailmeasure(C)
    &=\sum_{k\in\Zset} \int_0^\infty \pr\Big(r\shift^k \bTheta/ \norm{\shift^k\bTheta}_{q,\alpha} \in C\Big)\alpha r^{-\alpha-1}\rmd r\\
    &=\sum_{k\in\Zset} \int_0^\infty \pr\Big(\bTheta \in C\Big)\alpha r^{-\alpha-1}\rmd r =\pr(\bTheta\in C)\times \infty  \; .
  \end{align*}
  This proves the second statement.
\end{proof}

\subsection{Another representation of the tail measure}
\label{sec:other-construction}

We propose here another construction proof of
\Cref{theo:equivalent-tailmeasure-spectraltailprocess}. It is based on the infargmax functional $I$
defined on $\seqspace\setE$ by
\begin{align*}
  I(\bx) =
  \begin{cases}
    - \infty & \mbox{ if } \limsup_{k\to-\infty} \pseudonorm{\bx_k}[E] =\sup_{k\in\Zset} \pseudonorm{\bx_k}[E] \; , \\
    j \in \Zset & \mbox{ if } \sup_{k\leq j-1}\pseudonorm{\bx_k}[E] < \pseudonorm{\bx_j}[E]=\sup_{k\in\Zset} \pseudonorm{\bx_k}[E] \; , \\
    +\infty & \mbox{if} \sup_{k\leq j}\pseudonorm{\bx_k}[E] < \sup_{k\in\Zset} \pseudonorm{\bx_k}[E] \mbox{ for all $j$}\;.
  \end{cases}
\end{align*}
For $\bx\in\seqspace\setE\setminus\{\zeroseq\setE\}$, a sufficient condition for $I(\bx)\in\Zset$ is
$\lim_{|k|\to\infty} \pseudonorm{\bx_k}[E]=0$.

For two sequences $q\in[0,\infty)^\Zset$ and $\bx\in\setE^\Zset$, we define the pointwise
multiplication $q\cdot \bx$ by $(q\cdot \bx)_k = q_k\bx_k$, $k\in\Zset$.
\begin{proposition}
  \label{prop:infargmax-construction}
  Let $\bTheta$ be a process which satisfies the time change
  formula~(\ref{eq:time-change-formula-0}) and let $q\in(0,\infty)^\Zset$ be such that
  $\pr(I(q\cdot\bTheta)\in\Zset)=1$.  Define the measure $\tailmeasure_q$ on $\seqspace\setE$ by
  \begin{align*}
    \tailmeasure_q(A) 
    & = \sum_{j\in\Zset} \int_0^\infty \pr\left(r\shift^j\bTheta\in A\;,\ I(q\cdot \shift^j\bTheta)=j\right) \alpha r^{-\alpha-1} \rmd r 
      \;,\quad A\in\mcf \; . 
  \end{align*}
  Then $\tailmeasure_q$ does not depend on $q$ and defines a shift-invariant tail
  measure with tail spectral process $\bTheta$.
\end{proposition}
Note that any $q\in(0,\infty)^\Zset$ such that $\sum_{k\in\Zset}q_k^\alpha<\infty$ satisfies
$\pr(I(q\cdot\bTheta)\in\Zset)=1$. Indeed, the time change formula implies that
$\esp[\pseudonorm{\bTheta_h}[E]^\alpha]\leq 1$, $h\in\Zset$, so that
$\esp[\sum_{j\in\Zset} q_j^\alpha \pseudonorm{\bTheta_j}[E]^\alpha]<\infty$ and therefore
$\lim_{|j|\to \infty} q_j\pseudonorm{\bTheta_j}[E]=0$ and $I(q\cdot\bTheta)\in\Zset$ almost surely.

\begin{proof}
  It is straightforward to check that $\tailmeasure_q$ is $\alpha$-homogeneous and satisfies
  $\tailmeasure_q(\{\zeroseq\setE\})=0$. For all measurable function
  $H:\seqspace\setE\to[0,\infty)$, we have
  \begin{align*}
    \int_{\seqspace\setE} 
    & H(\bx) \ind{\pseudonorm{\bx_0}[E]>1} \tailmeasure_q(\rmd\bx) \\
    & = \sum_{j\in\Zset} \int_0^\infty   \esp\left[ H(r\shift^j\bTheta) 
      \ind{r\pseudonorm{\bTheta_{-j}}[E]>1\;,\ I(q\cdot \shift^j\bTheta)=j} \right]   \alpha r^{-\alpha-1} \rmd r \\
    & = \sum_{j\in\Zset} \int_1^\infty   \esp\left[ \pseudonorm{\bTheta_{-j}}[E]^\alpha H\left( \frac{r B^j\bTheta} {\pseudonorm{\bTheta_{-j}}[E]} \right)      
      \ind{I(q\cdot B^j \bTheta)=j}   \right]   \alpha r^{-\alpha-1} \rmd r \\
    & = \sum_{j\in\Zset} \int_1^\infty   \esp\left[ H\left( \frac{r \bTheta}{\pseudonorm{\bTheta_0}[E]}\right)  \ind{I(q\cdot \bTheta)=j}   \right]   
      \alpha r^{-\alpha-1} \rmd r \\
    & = \int_1^\infty \esp\left[ H(r\bTheta) \sum_{j\in\Zset}  \ind{I(q\cdot \bTheta)=j} \right]  \alpha r^{-\alpha-1} \rmd r
      = \int_1^\infty \esp[H(r\bTheta)]  \alpha r^{-\alpha-1} \rmd r \; .
  \end{align*}
  We used here the definition of $\tailmeasure_q$ from
  Proposition~\ref{prop:infargmax-construction}, the change of variable $r'=r/\pseudonorm{\bTheta_{-j}}[E]$
  (note that the event $\{\pseudonorm{\bTheta_{-j}}[E]=0\}$ has no contribution to the integrals), the time
  change formula, the fact that $\pr(\pseudonorm{\bTheta_0}[E]=1)=1$ and finally the assumption
  $\sum_{j\in\Zset}\pr(I(q\cdot \bTheta)=j)=1$.
  
  At this point, we have retrieved \Cref{eq:proof-identity} and the remainder of the proof
  follows exactly the same lines as the proof of
  \Cref{theo:equivalent-tailmeasure-spectraltailprocess}.
\end{proof}

\begin{remark}
  \label{rem:other-choices}
  In the particular case $\pr(I(\bTheta)\in\Zset)=1$, we can take $q\equiv 1$ and we get
  \begin{align*}
    \tailmeasure(A) 
    = \sum_{j\in\Zset} \int_0^\infty \pr\left(r\shift^j\bTheta\in A\;,\ I(\bTheta)=0\right) \alpha r^{-\alpha-1} \rmd r\;,\quad A\in\mcf \; .
  \end{align*}
  Introducing the process $\bQ$ such that $\mathcal{L}(\bQ)=\mathcal{L}(\bTheta \mid I(\bTheta)=0)$,
  we obtain
  \begin{align*}
    \tailmeasure(A) 
    = \pr(I(\bTheta)=0) \sum_{j\in\Zset} \int_0^\infty \pr(r\shift^j\bQ\in A) \alpha r^{-\alpha-1} \rmd r\;,\quad A\in\mcf \; . 
  \end{align*}
  This representation is similar as the one from Eq.~\eqref{eq:dissipative-1}. In fact, this is a
  special case of a moving shift representation of $\tailmeasure$, see \Cref{sec:dissipative}.
\end{remark}

\subsection{Moving shift representations and dissipative tail measures}
\label{sec:dissipative}
We consider in this section the relationship between the existence of a moving shift representation and the dissipative/dissipative decomposition of a tail measure. Note that ergodic properties of tail measures are also considered in \cite{samorodnitsky:owada:2012}, section 5.
We introduce only the minimum amount of ergodic theory and define the notion of dissipative tail measure. For more details on (infinite measure) ergodic theory, we refer to \cite{A98}. The $\sigma$-field on $\seqspace\setE$ generated by cones, or equivalently by $0$-homogeneous functions, is denoted by  $\mathcal{C}$.
\begin{definition}\label{def:dissipative}
The dynamical system $(\seqspace\setE,\mathcal{C},\tailmeasure,\shift)$ is said dissipative if there exists a cone $C_0\in\mathcal{C}$ such that the sets $\shift^h C_0$, $h\in\mathbb{Z}$, are pairwise disjoint and $\tailmeasure$ is supported by $D=\bigcup_{h\in\Zset}\shift^h C_0$, that is $\tailmeasure\left(\seqspace\setE\setminus D\right)=0$.
\end{definition}

On the other hand, \Cref{rem:dissipative-1,rem:other-choices} above motivate the following definition.
\begin{definition}
  \label{def:diss-representation}
  We say that a shift-invariant tail measure $\tailmeasure$ has a moving shit representation if there exists a
  stochastic process $\tbZ$ such that
  \begin{equation}
    \label{eq:diss-representation}
    \tailmeasure(A)=\sum_{h\in\Zset} \int_{0}^\infty \pr\left(r\shift^h \tbZ\in A\right)\alpha r^{-\alpha-1}\rmd r \; , \ \ A\in\mcf \; .
  \end{equation}
\end{definition}
The conditions $\tailmeasure(\{\zeroseq\setE\})=0$ and $\tailmeasure(\{\pseudonorm{\bx_0}[E]>1\})=1 $ entail
\begin{equation}
  \label{eq:cond-tildeZ}
  \pr(\tbZ=\zeroseq\setE)=0 \; , \ \ \sum_{h\in\Zset}\esp[\pseudonorm{\tbZ_h}[E]^\alpha]=1 \; .
\end{equation}
Indeed, we have
\[
\tailmeasure(\{\pseudonorm{\bx_0}[E]>1\}) = \sum_{h\in\Zset} \int_0^\infty \pr(r\pseudonorm{\tbZ_{-h}}[E]>1)\alpha
r^{-\alpha-1} \rmd r = \sum_{h\in\Zset}\esp[\pseudonorm{\tbZ_h}[E]^\alpha] \; .
\]
Conversely, it is easily proved that, for any stochastic process $\tbZ$ satisfying
\eqref{eq:cond-tildeZ}, the measure $\tailmeasure$ defined by \eqref{eq:diss-representation} is a
shift-invariant tail measure.

\begin{remark}
  \Cref{def:diss-representation} is strongly related to the notion of mixed moving maximum
  representation for max-stable process. If a max-stable process $\bX$ has a dissipative exponent
  measure with representation~\eqref{eq:diss-representation}, then it can be represented as
  \[
  \bX_h\stackrel{d}=\bigvee_{i\geq 1} U_i \tbZ^{(i)}_{h-T_i},\quad h\in\Zset \,,
  \]
  where $\sum_{i\geq 1}\delta_{(U_i,T_i)}$ is a Poisson random measure on $(0,\infty)\times \Zset$
  with intensity equal to the product of $\alpha u^{-\alpha-1}\mathrm{d}u$ with the counting measure
  on $\Zset$, and, independently, $\tbZ^{(i)}$ are independent copies of $\tbZ$. This is a mixed
  moving maximum representation and $\bX$ is generated by a dissipative flow \citep[Theorem
  8]{DK2016}.
\end{remark}

\begin{remark}
  \label{rem:Z-tildeZ}
  \Cref{theo:nu-to-X} states that any tail measure has a stochastic
  representation~\eqref{eq:polar-nu-Z}. One can wonder what is a stochastic representation for a
  tail measure $\tailmeasure$ given by a moving shift representation \eqref{eq:diss-representation}. A
  possible construction is as follows: starting from $\tbZ$, consider an independent $\Zset$-valued
  random variable $K$ such that $p_k=\pr(K=k)\in (0,1)$, $k\in \Zset$ and define 
  \begin{align}
    \label{bemask}
    \bZ= \sum_{k \in \Zset} p_k^{-1/\alpha} B^k \tbZ\, \ind{K=k} \; .  
  \end{align}
  In this construction $\bZ$ appears as a randomly shifted and rescaled version of $\tbZ$. It is
  easy to check that the stochastic representation \eqref{eq:polar-nu-Z} and the dissipative
  representation \eqref{eq:diss-representation} define the same tail measure $\tailmeasure$.
\end{remark}

The converse is not true, that is a shift-invariant tail measure does not always have a moving shift
representation of the form~\eqref{eq:diss-representation}. The next result is strongly related to
\cite[Theorem~3]{DK2016}. We say that $\tailmeasure$ (resp. $\bZ$, $\bTheta$) is supported by
$A\in\mcf$ if $\tailmeasure(A^c)=0$ with $A^c$ the complement of $A$ in $\seqspace\setE$ (resp.
$\pr(\bZ\in A^c)=0$, $\pr(\bTheta\in A^c)=0$).
\begin{proposition}
  \label{prop:dissipative}
  Let $\tailmeasure$ be a shift-invariant tail measure. The following statements are equivalent:
  \begin{enumerate}[(i),wide=0pt]
  \item \label{item:dissipative}  $(\seqspace\setE,\mathcal{C},\tailmeasure,\shift)$ is dissipative; 
  \item \label{item:movingshift} $\tailmeasure$ has a moving shift representation \eqref{eq:diss-representation};
  \item \label{item:support-summable}   $\tailmeasure$ is supported by
   $\{\bx:\sum_{h\in\Zset}\pseudonorm{\bx_h}^\alpha<\infty\}$;
  \item \label{item:support-0}   $\tailmeasure$ is supported by $\{\bx:\lim_{|h|\to\infty}\pseudonorm{\bx_h}[E]=0\}$;
  \item \label{item:support-infargmax}   $\tailmeasure$ is supported by $\{\bx:I(\bx)\in\Zset\}$.
  \end{enumerate}
\end{proposition}

\begin{proof}
  \begin{enumerate}[-,wide=0pt]
  \item \ref{item:dissipative}~$\Rightarrow$~\ref{item:movingshift}: let $C_0$ be
      as in \Cref{def:dissipative}. According to \Cref{theo:tiltshift}, there exists an
      $\setE$-valued stochastic process $\bZ$ which satisfies~(\ref{eq:cond-Z})
      and~(\ref{eq:polar-nu-Z}). Therefore, the restriction $\tailmeasure_0$ of the tail measure
      $\tailmeasure$ to $C_0$ can be represented as
      \[
      \tailmeasure_0(A)=\int_0^\infty \mathbb{P}(r\bZ\in A)\alpha r^{-\alpha-1}\rmd r  \; ,
      \]
      for all measurable sets $A\subset C_0$.  The fact that $\tailmeasure$ is dissipative implies
      that $\tailmeasure=\sum_{h\in\Zset}\tailmeasure_0\circ \shift^{-h}$ and hence that
      $\tailmeasure$ admits the representation \eqref{eq:diss-representation} with $\tbZ=\bZ$.  
  \item \ref{item:movingshift}~$\Rightarrow$~\ref{item:support-summable}: If $\tailmeasure$ has a
    dissipative representation \eqref{eq:diss-representation}, then $\tbZ$ satisfies
    \eqref{eq:cond-tildeZ} and $\mathbb{E}[\sum_{h\in\Zset}\pseudonorm{\tbZ_h}[E]^\alpha]<\infty$ implies that
    $\tbZ$ is supported by $\{\bx:\sum_{h\in\Zset}\pseudonorm{\bx_h}[E]^\alpha<\infty\}$. Then, the
    representation~(\ref{eq:diss-representation}) implies that this set also supports
    $\tailmeasure$.
  \item
    \ref{item:support-summable}~$\Rightarrow$~\ref{item:support-0}~$\Rightarrow$~\ref{item:support-infargmax}:
    these implications are trivial since $\sum_{h\in\Zset}\pseudonorm{\bx_h}[E]^\alpha<\infty$ implies
    $\lim_{|h|\to\infty}\pseudonorm{\bx_h}[E]=0$, which in turn implies $I(\bx)\in\Zset$ for
    $\bx\neq \zeroseq\setE$ (recall $\tailmeasure (\{\zeroseq\setE\})=0$).
  \item \ref{item:support-infargmax}~$\Rightarrow$~\ref{item:dissipative}: take $C_0=\{\bx: I(\bx)=0\}$ to check that $\tailmeasure$ is dissipative.
  \end{enumerate}
\end{proof}

\begin{remark}
  Since the sets $\{\bx:\sum_{h\in\Zset}\pseudonorm{\bx_h}[E]^\alpha<\infty\}$,
  $\{\bx:\lim_{|h|\to\infty}\pseudonorm{\bx_h}[E]=0\}$ and $\{\bx:I(\bx)\in\Zset\}$ are shift-invariant
  cones, \Cref{lem:support-cone} implies that \ref{item:support-summable}, \ref{item:support-0} and
  \ref{item:support-infargmax} can be equivalently expressed with $\bZ$ or $\bTheta$ where $\bZ$ is a
  stochastic representation of $\tailmeasure$ as in \eqref{eq:polar-nu-Z} and $\bTheta$ is the corresponding 
  spectral tail process. 
\end{remark}

\subsection{Maximal indices}
\label{sec:maxindices}
We introduce in this section the maximal indices of a shift-invariant tail measure $\tailmeasure$
that are closely connected with the extremal indices of regularly varying stationary time series,
see \Cref{sec:extremal-indices} below.

Given an $\alpha$-homogeneous shift-invariant tail measure $\tailmeasure$ and a $1$-homogeneous
functional~$\tau:\seqspace{\setE}\to [0,\infty]$ such that $\tailmeasure(\{\tau(\bx)>1\})=1$ , we
define the quantity $\theta_\tau\in [0,1]$, called maximal index, by 
\begin{equation}
  \label{eq:def-max-ind}
  \theta_\tau 
  = \lim_{n\to \infty} \frac{1}{n} \tailmeasure\left(\left\{\max_{0\leq h\leq n-1}\tau(\shift^{h} \bx)>1  \right\}\right) \; . 
\end{equation}
The existence of the limit is a consequence of Fekete's subadditive lemma. The shift invariance of
$\tailmeasure$ implies that the sequence
$u_n=\tailmeasure\left(\left\{\max_{0\leq h \leq n-1}  \tau(\shift^{h} \bx)>1 \right\} \right)$, $n\geq 1$, is
subadditive. As a consequence, $u_n/n$ converge to $\inf_{n\geq 1}u_n/n$ and the limit is in $[0,1]$
since the sequence is non-negative and $u_1=1$.

The next result shows that the maximal indices of a dissipative tail measure are positive and
provides expressions of the maximal indices in terms of the stochastic representation and the
spectral tail process of the tail measure.
\begin{proposition}
  \label{prop:max-indices}
  Assume that $\tailmeasure$ is dissipative and that the $1$-homogeneous measurable function
  $\tau:\seqspace{\setE}\to [0,\infty]$ satisfies $\tailmeasure(\{\tau(\bx)>1\})=1$. Then
  $\theta_\tau>0$ and
  \begin{align*}
    \theta_\tau 
    & = \esp\left[ \sup_{h\in \Zset} \tau^\alpha( B^{h} \widetilde \bZ )\right]
     = \esp\left[\frac{ \sup_{h \in \Zset } \tau^\alpha( B^{h} \bTheta )}{\sum_{h \in \Zset} \pseudonorm{\bTheta_h}[E]^\alpha}\right] 
     = \pr(I(\bTheta)=0)\esp\left[ \sup_{h\in \Zset} \tau^\alpha( B^{h} \bQ )\right] \; , 
  \end{align*}
  with $\tbZ$ as in the  dissipative representation \eqref{eq:diss-representation}, $\bTheta$ the
  spectral tail process of $\tailmeasure$ and~$\bQ$ is a random sequence in $\seqspace\setE$ with
  distribution $\mathcal{L}(\bTheta \mid I(\bTheta)=0)$ as in Remark~\ref{rem:other-choices}.
\end{proposition}

\begin{remark} 
  \label{rem:candidate}
  For a dissipative tail measure $\tailmeasure$ and $\tau(\bx)=\pseudonorm{\bx_0}[E]$, we also have the
  following identity proved in \cite[Lemma~3.2]{planinic:soulier:2017}
  \begin{align*}
    \theta_\tau = \pr(\sup_{i\geq1} \pseudonorm{\bY_i}[E] \leq 1)  = \pr(\sup_{i\geq1} \pseudonorm{\bY_i}[E] > 1) \; ,
  \end{align*}
  where $\bY_i=Y\bTheta_i$, $i \in \Zset$ and $Y$ is a Pareto random variable with tail index
  $\alpha$, independent of the sequence $\{\bTheta_j\}$.  This means that the maximal
  index is in this case the candidate extremal index introduced in \cite{basrak:segers:2009}. The
  link with the usual extremal index will be made in \Cref{sec:regvar-ts}.
\end{remark}

The proof of \Cref{prop:max-indices} makes use of the following identity due to
\cite[Lemma~3.2]{smith:weissman:1996}: for a summable  sequence
$(u_h)_{h\in\Zset}\in \seqspace{[0,\infty)}$, 
  \begin{equation}
    \label{eq:lem-max-indices}
    \lim_{n\to\infty} \frac{1}{n}\sum_{h\in \Zset} \max_{0\leq k\leq n-1} u_{h+k}= \sup_{h\in\Zset} u_h\;.
  \end{equation}

\begin{proof}[Proof of Proposition~\ref{prop:max-indices}]
  Since $\tailmeasure$ is dissipative, we can introduce a dissipative representation
  \eqref{eq:diss-representation} and write
  \begin{align*}
    \tailmeasure\left(\max_{0\leq k\leq n-1} \tau(\shift^k \bx)>1 \right)
    & = \sum_{h\in\Zset}\int_0^\infty \mathbb{P}\left( r\max_{0\leq k\leq n-1} \tau(\shift^{k+h} \tbZ)>1\right) \alpha r^{-\alpha-1}\rmd r\\
    & = \sum_{h\in\Zset}\mathbb{E}\left[ \max_{0\leq k\leq n-1} \tau^\alpha(\shift^{k+h} \tbZ)\right] \; .
  \end{align*}
  For $n=1$, we have in particular
  $\sum_{h\in\Zset}\mathbb{E}\left[ \tau^\alpha(\shift^{h} \tbZ)\right]=1$ thanks to the normalizing
  condition $\tailmeasure(\tau(\bx)>1)=1$.  This proves that the sequence
  $u_h=\tau^\alpha(\shift^{h} \tbZ)$, $h\in\Zset$, is almost surely summable and
  (\ref{eq:lem-max-indices}) implies
  \[
  \lim_{n\to\infty} \frac{1}n\sum_{h\in\Zset} \max_{0\leq k\leq n-1} \tau^\alpha(\shift^{k+h} \tbZ)
  = \sup_{h\in\Zset}\tau^\alpha(\shift^{h} \tbZ) \; ,\ \ \mbox{almost surely} \; .
  \]
  Furthermore, for all $n\geq 1$, the left hand side in the previous equation is bounded from above
  by $\sum_{h\in\Zset} \tau^\alpha(\shift^{h} \tbZ)$ which has finite expectation. Lebesgue 's
  dominated convergence theorem implies
  \begin{align*}
    \theta_\tau
    & = \lim_{n\to\infty} \frac{1}{n}\tailmeasure\left(\max_{0\leq k\leq n-1} \tau(\shift^k \bx)>1 \right) \\
    & = \lim_{n\to\infty} \frac{1}{n}\sum_{h\in\Zset}\mathbb{E}\left[ \max_{0\leq k\leq n-1} \tau^\alpha(\shift^{k+h} \tbZ)\right]
      = \mathbb{E}\left[ \sup_{h\in\Zset}\tau^\alpha(\shift^{h} \tbZ)\right] \; .
\end{align*}
This proves the first formula.  The second and third expressions of $\theta_\tau$ are special cases
obtained for $\tbZ=\bTheta/(\sum_{k\in\Zset}\pseudonorm{\bTheta_k}[E]^\alpha)^{1/\alpha}$ and
$\tbZ=\pr^{1/\alpha}(I(\bTheta)=0) \bQ$, see Remarks~\ref{rem:dissipative-1}
and~\ref{rem:other-choices}.
\end{proof} 

\section{Regularly varying time series on a metric space}
\label{sec:regvar-ts}
In this section, we will build a regularly varying time series with a prescribed tail measure. For
this purpose, we first recall the most important definitions and properties of $\mcm_0$ convergence
and regular variation on a metric space.
For the sake of clarity, the results are stated for a general metric space $F$ in section 3.1 and 3.2 and we consider  the specific case $\setF=\seqspace\setE$ in later sections.

\subsection{Regular variation on a metric space}
We follow here \cite[Section~3]{hult:lindskog:2006}. Let $(\setF,\dist)$ be a metric space and
let~$\zero{F}$ be an element of $\setF$. We assume that there exists a \textcolor{blue}{continuous} map $(s,\bx)\to s\bx$ from
$[0,\infty)\times\setF$ to $\setF$ such that for all $\bx\in\setF$ and $s\leq t\in(0,\infty)$,
$s(t\bx)=(st)\bx$, $0\bx=\zero\setF$ and
\begin{align*}
  \dist(\zero{F},s\bx) \leq  \dist(\zero{F},t\bx) \; .
\end{align*}
Such a map will be called a distance compatible outer multiplication. We denote the ball with center
at $\zero{F}$ and radius $r\geq0$ by $B_r$.  We endow $\setF$ with its Borel $\sigma$-field.

Let $\mcm_0(\setF)$ be the set of boundedly finite measures on $\setFzero$, that is measures $\nu$
such that $\nu(A)<\infty$ for all measurable sets $A$ such that $A\cap B_r=\emptyset$ for some
$r>0$. Such sets will be called separated from $\zero{F}$. The null measure will be denoted
by $\nullmeasure$.  We will say that a sequence $\{\nu_n,n\geq1\}$ of measures in $\mcm_0(\setF)$
converges in $\mcm_0(\setF)$ to a measure $\nu$, which we will denote by
$\nu_n\convmzero{\setF}\nu$, if
\begin{align*}
  \lim_{n\to\infty} \nu_n(A) = \nu(A) \; , 
\end{align*}
for all measurable set $A$ separated from $\zero{F}$ and such that $\nu(\partial A)=0$. This type of
convergence is referred to as weak$^\#$ convergence in \cite{daley:vere-jones:bookvolI} and simply
vague convergence in \cite{kallenberg:2017}. For more details on the relationship between these different types of convergence, we refer to \cite{LRR14} or \cite{BP18}.

By \cite[Lemma~4.1]{kallenberg:2017},
$\lim_{n\to\infty} \nu_n\convmzero{\setF}\nu$ if and only if $\lim_{n\to\infty}\nu_n(f)=\nu(f)$ for
all bounded Lipschitz continuous functions with support separated from
zero. \cite{hult:lindskog:2006} proved that convergence in $\mcm_0(\setF)$ is equivalent to weak
convergence on the complement of balls centered at $\zero{F}$. More precisely,
\begin{align}
  \nu_n\convmzero{\setF} \nu \Longleftrightarrow \mbox{ for all but countably many } r >0 \; , 
  \ \ {\nu_n}_{\mid_{B_r^c}} \stackrel{w}{\longrightarrow} \nu_{\mid_{B_r^c}}
\end{align}
where $\nu_{\mid_A}$ is the measure $\nu$ restricted to the set $A$ and $\stackrel{w}\to$ denotes
weak convergence. Convergence in $\mcm_0$ can be metrized. Let $\rho_r$ be Prohorov's distance on
the set of finite measures defined on $B_r^c$. Let $\rho$ be the metric on $\mcm_0(\setF)$ defined
by:
\begin{align}
  \label{eq:def-dist-rho}
  \rho(\mu,\nu) = \int_0^\infty \rme^{-t} (\rho_r(\mu,\nu) \wedge1)  \rmd r \; , \ \ \ \mu,\nu\in\mcm_0(\setF)\; .
\end{align}
Then $(\mcm_0(\setF),\rho)$ is a complete separable metric space;
cf. \cite[Theorem~2.3]{hult:lindskog:2006}.

We can now define regular varying measures and random elements in $\setF$.
\begin{definition}
  \label{def:regvar-metric}
  \begin{itemize}
  \item A Borel measure $\mu$ on $\setF$ is said to be regularly varying if there exists a non
    decreasing sequence $\{a_n\}$ and a measure $\mu^*\in\mcm_0(\setF)$ such that
    $n\mu(a_n\cdot) \convmzero\setF \mu^*$. We then write
    $\mu\in\mathrm{RV}(\setF,\{a_n\},\mu^*)$.
  \item  An $\setF$-valued random element $\bX$ defined on a probability space
    $(\Omega,\mca,\pr)$ is said to be regularly varying if there exists a non decreasing
    sequence $\{a_n\}$ tending to infinity and a nonzero measure $\nu$ on $\setFzero$ such that
    $n\pr(a_n^{-1}\bX\in\cdot)\convmzero\setF\nu$. We then write
    $\bX\in\mathrm{RV}(\setF,\{a_n\},\nu)$.
  \end{itemize}
\end{definition}
By \cite[Theorem~3.1]{hult:lindskog:2006}, if $\bX\in\mathrm{RV}(\setF,\{a_n\},\nu)$, then there
exists $\alpha>0$ which will be called the tail index of $\bX$ such that the measure $\nu$ is
$\alpha$-homogeneous and the sequence $\{a_n\}$ is regularly varying with index $1/\alpha$.  We will
need the following result which is a straightforward application of the mapping theorem
\cite[Theorem~2.5]{hult:lindskog:2006}. 

\begin{lemma}
  \label{lem:transfer-homogeneous-regvar}
  Let $(\setF,d)$ and $(\setF',d')$ be two complete separable metric spaces each endowed with a
  distance compatible outer multipication. Let $\zero\setF\in\setF$ and let $T:\setF\to\setF'$ be a
  1-homogeneous map such that $T(\zero\setF)=\zero{\setF'}$.  Set
  $\setF_0 = \setF\setminus\{\zero\setF\}$ and $\setF_0' = \setF'\setminus\{\zero{\setF'}\}$.  Let
  $\mu,\mu^*$ be a Borel measures on $\setF$ and let $\{a_n\}$ be a non decreasing sequence such
  that $\mu\in \mathrm{RV}(\setF_0,a_n,\mu^*)$. If $T$ is $\mu^*$ almost
  surely continuous,  continuous at $\zero\setF$, and $\mu^*\circ T^{-1}$ is not the null measure, then
  $\mu\circ T^{-1} \in \mathrm{RV}(\setF'_0,a_n,\mu^*\circ T^{-1})$.
\end{lemma}

\begin{proof}
  Define $\mu_n = n\mu(a_n\cdot)$. By assumption, $\mu_n\convmzero\setF\mu^*$. By homogeneity of
  $T$, $\mu_n\circ T^{-1} =n \mu\circ T^{-1}(a_n\cdot)$.  We want to apply
  \cite[Theorem~2.5]{hult:lindskog:2006} to prove that
  $\mu_n\circ T^{-1}\convmzero{\setF'}\mu^*\circ T^{-1}$. Since $T(\zero\setF) = \zero{\setF'}$,
  there only remain to prove that if $A$ is bounded away from $\zero{\setF'}$, then $T^{-1}(A)$ is
  bounded away from $\zero\setF$.  If $A\subset \setF'$ is bounded away from $\zero{\setF'}$, there
  exists $\epsilon>0$ such that $\by\in A$ implies $\dist(\by,\zero{\setF'})>\epsilon$. Since $T$ is
  continuous and $T(\zero\setF)=\zero{\setF'}$, there exists $\eta>0$ such that
  $\dist(\bx,\zero\setF)\leq \eta$ implies $\dist(T(\bx),\zero{\setF'})\leq\epsilon$. This proves
  that if $\bx\in T^{-1}(A)$ then $\dist(\bx,\zero\setF)>\eta$.
\end{proof}

\subsection{Regular varying Poisson point processes}

Let $\mcn_0(\setF)$ be the set of boundedly finite point measures on $\setFzero$, \ie\ measures
$\nu$ such that $\nu(A)\in\Nset$ for all bounded Borel set $A$ separated from $\zero\setF$. This
implies that $\nu$ has a finite number of points outside each ball centered at $\zero\setF$ and we
can write $\nu = \sum_{j\geq1} \delta_{\bx_j}$ where the points of $\nu$ are numbered in such a way that
$$
d(\zero\setF,\bx_i) \geq d(\zero\setF,\bx_j)
$$ 
if $i\leq j$.  It is then easily seen that $\mcn_0(\setF)$ is a closed subset of $\mcm_0(\setF)$ and that the 
convergence  $\nu_n\convmzero{\setF}\nu$ implies the convergence of points in $\setF$.

The restriction of the distance $\rho$ defined in~(\ref{eq:def-dist-rho}) to the space
$\mcn_0(\setF)$ has the following property. Let the null measure be denoted by $\nullmeasure$ and
let $\pi\in\mcn_0(\setF)$.  Let the largest distance of a point of $\pi$ to $\zero\setF$ be denoted
by $\largestpoint\pi[F]$, \ie\
\begin{align*}
  \largestpoint{\pi}[F] = \sup_{\bx\in\pi} d(\zero{\setF},\bx) \; . 
\end{align*}
If $r>\largestpoint{\pi}[F]$, then $\pi$ has no point outside $B_r$ and thus
$\rho_r(\nullmeasure,\pi)=0$. Moreover, by definition of the Prohorov distance,
\begin{align*}
  \rho_r(\nullmeasure,\pi) = \inf\{\alpha>0: \pi(F\cap B_r^c)\leq\alpha, \ F \mbox{ closed}\} = \pi(B_r^c) \; .
\end{align*}
That is, the Prohorov distance of a point measure to the zero measure is its number of points. 
Therefore, if $r>\largestpoint{\pi}[F]$, then   $\rho_r(\nullmeasure,\pi) =0$.  This yields
\begin{align}
  \label{eq:majoration-prohorov-nullmeasure}
  \rho(\nullmeasure,\pi) 
  & = \int_0^{\largestpoint{\pi}[F]} \rme^{-r} (\rho_r(\nullmeasure,\pi) \wedge1)  \rmd r  \leq \largestpoint{\pi}[F] \; . 
\end{align}
On the other hand, if $r<\largestpoint{\pi}[F]$ then $\rho_r(\nullmeasure,\pi)\geq1$ and
$1-\rme^{-x}\geq (x\wedge1)/2$, thus we have
\begin{align}
  \label{eq:minoration-prohorov-nullmeasure}
  \rho(\nullmeasure,\pi) 
  & \geq  \int_0^{\largestpoint{\pi}[F]}   \rme^{-r} \rmd r = (1-\rme^{-\largestpoint{\pi}[F]}) 
    \geq \frac{1}2 (\largestpoint{\pi}[F]\wedge1) \; . 
\end{align}
These bounds imply that a subset $A\subset \mcn_0(\setF)$ is separated from $\nullmeasure$ if there
exists $\epsilon>0$ such that $\largestpoint{\pi}[F]>\epsilon$ for all $\pi\in A$.

We define the mutiplitcation $(t,\nu)\to t\cdot\nu$ for $t\in(0,\infty)$
and $\nu\in\mcm_0(\setF)$ by
\begin{align*}
  t\cdot\nu(f) = \int_{\setE} f(t\bx) \nu(\rmd \bx)
\end{align*}
for all nonnegative measurable functions $f$. If $\nu=\sum_{j\geq1} \delta_{\bx_i}$ is a point
measure, then $t\cdot\nu=\sum_{j\geq1}\delta_{t\bx_j}$.  Multiplication is continuous with respect
to the product topology. For $\pi\in\mcn_0$ and $0 < s < t$, 
\begin{align*}
  \rho_r(\nullmeasure, s\pi) = \pi(s^{-1}B_r^c) \leq \pi(s^{-1}B_r^c)  =   \rho_r(\nullmeasure, s\pi) \; .
\end{align*}
Therefore we can define a regularly varying point process on $\setFzero$ as a regularly varying
element in $\mcn_0(\setF)$ in the sense of \Cref{def:regvar-metric}.

\begin{theorem}
  \label{theo:RV-N0}
  Let $\mu_0,\mu \in\mcm_0(\setF)$ and $\{a_n\}$ be a nondecreasing sequence such that
  $a_n\to\infty$ and $n\mu_0(a_n\cdot) \convmzero{\setF} \mu$ as $n\to\infty$. Let $\Pi$ be a
  Poisson point measure on $\setFzero$ with mean measure
  $\mu_0$. 
  Then $\Pi\in\mathrm{RV}(\mcn_0(\setF),\{a_n\},\mu^*)$ where $\mu^*$ is a measure on
  $\mcn_0(\setF)\setminus\{\nullmeasure\}$ defined by
  \begin{align*}
    \mu^*(B) = \int_{\setF} \ind{\delta_{\bx}\in B} \mu(\rmd \bx) \; , 
  \end{align*}
  for all Borel set $B$ of $\mcn_0(\setF)$ endowed with the distance $\rho$, and $\delta_{\bx}$
  denotes the Dirac mass at $\bx\in\setE$.  If $\mu$ is $\alpha$-homogeneous and $\Pi\sim PPP(\mu)$,
  then $\Pi\in\mathrm{RV}(\mcn_0(\setF),n^{1/\alpha},\mu^*)$.
\end{theorem}

Note that the limit measure $\mu^\ast$ is the image of $\mu$ under the injection of $\setF$ into
$\mcn_0(\setF)$ defined by $\bx\mapsto \delta_{\bx}$. It is concentrated on the subset of point
measures that have exactly one point. The underlying heuristic is that given that $\Pi$ is large (in
the sense $d(0,\Pi)>u$ with $u\to\infty$), then $\Pi$ can be approximated by a random point measure
with only one large point. This is yet another instance of the so-called single large jump
principle.

\begin{proof}
  We need to prove the convergence
  \begin{equation}
    \label{eq:M0conv-PPP}
    n\pr(\Pi/a_n\in \cdot ) \convmzero{\mcn_0(\setF)}  \mu^{\ast}  \; .   
  \end{equation}
  By \Cref{theo:Laplace}, the convergence~(\ref{eq:M0conv-PPP}) holds if
  \begin{equation}
    \label{eq:M0conv-PPP-Laplace}
    \lim_{n\to\infty}    n\left(\esp\left[1-\rme^{-\int_{\setF}f(\bx/a_n)\Pi(\rmd \bx)}\right]\right) = 
    \int_{\mcn_0(\setF)} \left(1-\rme^{-\int_{\setF}f(\bx)\pi(\rmd \bx)}\right)\mu^\ast(\rmd \pi) \; ,
  \end{equation}
  for all continuous function $f:\setF\to [0,\infty)$ vanishing on a neighborhood of $\bszero_\setF$.  By
  definition of $\mu^\ast$, the right-hand side of~(\ref{eq:M0conv-PPP-Laplace}) is equal to 
  \[
  \int_{\mcn_0(\setF)} (1-\rme^{-\int_{\setF}f(\bx)\pi(\rmd\bx)})\mu^\ast(\rmd \pi)=\int_\setF
  (1-\rme^{-f(\bx)})\mu(\rmd\bx).
  \]
  On the other hand, since $\Pi$ is a Poisson point process, we have 
  \begin{align*}
    n\left(\esp\left[1-\rme^{-\int_{\setF}f(\bx/a_n)\Pi(\rmd\bx)}\right]\right)
    & = n\left(1-\exp\left[\int_{\setF}\left(\rme^{-f(\bx/a_n)}-1\right)\mu_0(\rmd\bx)\right]\right) \\
    & = n\left(1-\exp\left[n^{-1}\int_{\setF}-\left(1-\rme^{-f(\bx)}\right)\mu_n(\rmd\bx)\right]\right) \; , 
  \end{align*}
  with $\mu_n=n\mu_0(a_n\cdot)$. The function $1-\rme^{-f}$ is non negative, bounded and with
  support separated from zero; moreover $\mu_n \to \mu$ in $\mcm_0$ by assumption, therefore
  \begin{align*}
 \lim_{n\to\infty}   n\left(\esp\left[1-\rme^{-\int_{\setF}f(\bx/a_n)\Pi(\rmd\bx)}\right]\right)
    & = \lim_{n\to\infty} n \left(1-\exp\left[n^{-1}\int_{\setF}\left(\rme^{-f(\bx)}-1\right)\mu_n(\rmd\bx)\right]\right)\\
    & =   \int_{\setF}\left(1-\rme^{-f(\bx)}\right)\mu(\rmd\bx) \; .
\end{align*}
This proves the convergence \eqref{eq:M0conv-PPP-Laplace} and the claimed regular variation of $\Pi$. 
\end{proof}

\subsection{Regularly varying  time series}
\label{sec:regvar-ts-construction}

We now introduce the notion of a regularly varying  time series. We consider a complete
separable metric space $(\setE,\dist[\setE])$ with an element $\zero \setE$ and we assume that the
metrid $\dist[\setE]$ has the homogeneity property
$\dist[\setE](\zero\setE,s\bx) = s\dist[\setE](\zero\setE,\bx)$ for all $s>0$ and $\bx\in\setE$. We
then define the pseudo norm $\pseudonorm{\bx} [\setE] = \dist[\setE](\zero\setE,\bx)$.
\begin{definition}
  \label{def:regvar-ts}
  Let $\bX=\{X_j,j\in\Zset\}$ be a time series with values in $\setE$. It is said to be regularly
  varying if $(X_s,\dots,X_t)$ is regularly varying in $\setE^{t-s+1}$ for all $s\leq t \in\Zset$.
\end{definition}
\cite{samorodnitsky:owada:2012} proved that if $\bX$ is regularly varying, then there exists a
measure $\tailmeasure$ on $\setE^\Zset$, called the tail measure of $\bX$, whose finite dimensional
projections are the exponent measures $\tailmeasure_{s,t}$ and having the properties of a tail measure as
introduced in~\Cref{def:tail-measure}. If $\bX$ is stationary, then the tail measure is shift
invariant.

Consider the metric $\dist[F]$ on $\setF=\setE^\Zset$ defined by
\begin{align}
  \dist[F](\bx,\by) = \sum_{j\in\Zset} 2^{-|j|} (\dist[\setE](\bx_j,\by_j)\wedge1) \; . \label{eq:distF}
\end{align}
It is proved in \cite[Theorem~4.1]{segers:zhao:meinguet:2017} that the regular variation of the time
series $\bX$ in the sense of \Cref{def:regvar-ts} is equivalent to the regular variation of $\bX$
seen as a random element with values in the complete separable metric space $(\setF,\dist[F])$ in
the sense of \Cref{def:regvar-metric}, \ie\ $\bX\in \mathrm{RV}(\setF,\{a_n\},\tailmeasure)$ with $a_n$ such
that $\lim_{n\to\infty}n\pr(\pseudonorm{\bX_0}[E]>a_n)=1$. Therefore, we will hereafter
indifferently say that $\bX$ is regularly varying in the sense of \Cref{def:regvar-ts} with tail
measure $\tailmeasure$ or $\bX\in \mathrm{RV}(\seqspace{\setE},\{a_n\},\tailmeasure)$.

The local tail process and spectral tail process associated to the tail measure $\tailmeasure$ can
be reinterpreted as limiting quantities for the regularly varying time series $\bX$. Their existence
also characterizes regular variation. The next result generalizes
\cite[Theoreom~2.1]{basrak:segers:2009} for a  non stationary time series.
\begin{lemma}
  \label{lem:Theta-as-limit}
  Let $\tailmeasure$ be a tail measure on $\setE^\Zset$ and for $h\in\Zset$ set $p_h=\tailmeasure(\{\pseudonorm{\bx}[E]>1\})$. For $h$ such
  that $p_h>0$, let $\bY^{(h)}$ and $\bTheta^{(h)}$ be the local tail and spectral tail  processes associated to $\tailmeasure$ as in
  \Cref{def:local-tail-process}.  The following statements are equivalent;
  \begin{enumerate}[(i)]
  \item \label{item:XRV} $\bX\in \mathrm{RV}(\seqspace{\setE},\{a_n\},\tailmeasure)$; 
  \item \label{item:tailprocess} For all $h\in\Zset$, $\lim_{n\to\infty} n\pr(\pseudonorm{\bX_h}[E]>a_n)=p_h$ and for
    all $h$ such that $p_h>0$, we have, as $u\to\infty$,
    \begin{align}
      \label{eq:tail-process-time-series-tailmeasure}
      &  \mathcal{L}\left(\bX/u \ \Big|\ \pseudonorm{\bX_h}[E]>u \right) \stackrel{d}\longrightarrow \bY^{(h)} \; ;
    \end{align}
  \item \label{item:spectraltailprocess} For all $h\in\Zset$, $\lim_{n\to\infty} n\pr(\pseudonorm{\bX_h}[E]>a_n)=p_h$ and for $h$ such that
    $p_h>0$,
    \begin{align}
      &  \mathcal{L}\left(\bX/\pseudonorm{\bX_h}[E] \ \Big|\ \pseudonorm{\bX_h}[E]>u \right) \stackrel{d}\longrightarrow \bTheta^{(h)}\;.
    \end{align}
  \end{enumerate}
  If $\bX$ is stationary, then $\bTheta^{(h)} \equallaw \shift^h \bTheta$ and $\tailmeasure$ is
  shift-invariant.
\end{lemma}

\begin{proof}
  We start by proving the implication \ref{item:XRV}~$\Rightarrow$~\ref{item:tailprocess}. By
  definition of regular variation, for every $h\in\Zset$ we have
  $ \lim_{n\to\infty} n\pr(\pseudonorm{\bX_h}[E] > a_n) =p_h$ and for every set $A$ depending only
  on a finite number of coordinates, we have
  \begin{align*}
    \lim_{n\to\infty}  n\pr(\bX\in A , \pseudonorm{\bX_h}[E] > a_n) = \tailmeasure(\{A\cap \{\pseudonorm{\bx_h}[E]>1\}) \; .
  \end{align*}
  By definition of the local tail process, we obtain
  \begin{align*}
    \lim_{n\to\infty}  \pr(\bX\in A \mid  \pseudonorm{\bX_h}[E] > a_n) = \frac1{p_h}\tailmeasure(\{A\cap \{\pseudonorm{\bx_h}[E]>1\}) = \pr(\bY^{(h)}\in A) \; .
  \end{align*}
  To prove the converse implication \ref{item:tailprocess}~$\Rightarrow$~\ref{item:XRV}, we first
  note that the tail measure is characterized by its finite dimensional projections. Therefore it
  suffices to prove that these projections are characterized by the tail process. Let $A$ be a set
  which depends only on the coordinates between $s$ and $t$, $s\leq t \in\Zset$, and bounded away
  from $zero$ in $\setE^{t-s+1}$. This means that there exists $\epsilon>0$ such that $\bx\in A$
  implies that $\sum_{h=s}^t \ind{\pseudonorm{\bx}[E]>\epsilon}\geq1$. Note also that if $p_h=0$, then for
  all $\epsilon>0$, $ \lim_{n\to\infty} n\pr(\pseudonorm{\bX_h}[E] > a_n\epsilon) = 0$. Thus in the following
  computations we will omit the indices $h$ such that $p_h=0$. Decomposing according to the first
  exceedence over $\epsilon a_n$, we obtain
  \begin{align*}
    \tailmeasure_{s,t}(A) 
    & = \lim_{n\to\infty} n \pr(a_n^{-1}\bX_{x,t} \in A) \\
    & = \lim_{n\to\infty} \sum_{h=s}^t n \pr(a_n^{-1}\bX_{x,t} \in A, \pseudonorm{\bX_h}[E]>\epsilon, \max_{s\leq i <h} \pseudonorm{\bX_i}[E]\leq\epsilon) \\
    & = \lim_{n\to\infty} \sum_{h=s\atop p_h>0}^t n \pr(a_n^{-1}\pseudonorm{\bX_h}[E]>a_n\epsilon) 
      \frac{\pr(a_n^{-1}\bX_{x,t} \in A, \pseudonorm{\bX_h}[E]>\epsilon, \max_{s\leq i <h} \pseudonorm{\bX_i}[E]\leq\epsilon)}
      {\pr(a_n^{-1}\pseudonorm{\bX_h}[E]>a_n\epsilon) } \\
    & =      \sum_{h=s\atop p_h>0}^t \epsilon^{-\alpha} p_h \pr\left(\epsilon \bY_{s,t}^{(h)} \in A, 
      \max_{s\leq i \leq h-1} \pseudonorm{\bY_i}[E]\leq 1\right) \; .
  \end{align*}
  This proves that the finite dimensional distributions of the tail process characterize the tail
  measure. The proof of the equivalence
  \ref{item:tailprocess}~$\Leftrightarrow$~\ref{item:spectraltailprocess} is straightforward
  generalization of the corresponding result for $\Rset^d$ valued time series in
  \cite{basrak:segers:2009} and is omitted.
\end{proof}

\begin{remark}
  \label{rem:maxstable}
  In the case $\setE = [0,\infty)$, \Cref{lem:Theta-as-limit} implies that the tail measure of
    a time series $\bX\in \mathrm{RV}(\seqspace{[0,\infty)},\{a_n\},\tailmeasure)$ is the exponent
    measure of the limiting max-stable process, see Remarks~\ref{rem:max-stable1} and~\ref{rem:max-stable2}. More precisely, let $\bX^{(i)}$, $i\geq1$, be \iid\
    copies of $\bX$. Then the regular variation of $\bX$ implies that
    \begin{align*}
      a_n^{-1}\bigvee_{i=1}^n  \bX^{(i)} \fidi  \bigvee_{i=1}^\infty \mathbf{P}^{(i)}
    \end{align*}
    where the suprema are taken componentwise and $\sum_{i=1}^\infty \delta_{\mathbf{P}^{(i)}}$ is a
    Poisson point process on $\seqspace{[0,\infty)}$ with mean measure $\tailmeasure$. This also
    shows that for a max-stable process the tail measure and the exponent measure are the same.
\end{remark}

In the sequel, given a shift-invariant tail measure, or equivalently given a spectral tail process,
we will build a time series

\subsubsection{Construction of a stationary regularly varying time series }
\label{sec:construction}
As seen in \Cref{sec:tailmeasurerepresenation}, the tail measure of a stationary regularly varying
time series with tail index $\alpha>0$ is a shift-invariant tail measure with homogeneous with index $\alpha$. A
natural question is whether any shift-invariant tail measure $\tailmeasure$ on $\setF=\seqspace\setE$ is
the tail measure of a stationary regularly varying time series $\bX$. The purpose of this section is
to prove that the answer is positive and provide one construction for such a process $\bX$.

Our intuition is guided by the case $\setE=[0,\infty)$. Then, given a tail measure $\tailmeasure$ on
$\seqspace{[0,\infty)}$, the max-stable process $\bX$ with exponent measure~$\tailmeasure$ is
regularly varying with tail measure~$\tailmeasure$. Furthermore, $\bX$ is stationary if and only if
$\tailmeasure$ is shift-invariant. This provides a straightforward solution in the non negative
case. Before we generalize it, we recall the Poisson point process representation of the max-stable
process $\bX$: if $\tailmeasure$ admits representation \eqref{eq:polar-nu-Z} with $\bZ$ a
non-negative time series, then
\[
\bX \equallaw \bigvee_{i\geq 1} \Gamma_i^{-1/\alpha} \bZ^{(i)} \; , 
\]
where $\{\Gamma_i\}_{i\geq 1}$ are the points of a homogeneous Poisson process on $[0,\infty)$ and
independently, $\bZ^{(i)}$, $i\geq 1$, are independent copies of $\bZ$ and the supremum is taken componentwise.

In the general framework where $\setE$ is a complete separable metric space and $\tailmeasure$ is a
tail measure on $\setF=\seqspace\setE$, we consider a Poisson point process
$\Pi\sim \mathrm{PPP}(\tailmeasure)$. Note $\Pi$ can be constructed as
\begin{equation}
  \label{eq:representation-Pi}
  \Pi=\left\{ \Gamma_i^{-1/\alpha}\bZ^{(i)} \; , \ i\geq 1\right\} \; . 
\end{equation}

We interpret the point process $\Pi$ as a particle system that evolves in time, the $i$-th particle
having position $\varphi^{(i)}_h=\Gamma_i^{-1/\alpha}\bZ^{(i)}_h$ at time $h$. The random process
$\varphi^{(i)}=\Gamma_i^{-1/\alpha}\bZ^{(i)}\in\setF=\seqspace\setE$ is hence the trajectory of the
$i$-th particle. We construct a time series $\bX$ that records at each time $h$ the position of the
particle which is farthest away from $\zero\setE$, which we will call the largest point. More
formally, we define
\begin{equation}
  \label{eq:def-X-argmax}
  \bX_h=\mathbf{\varphi}^{(i_h)}_h \; , \ \ \  i_h=\mathop{\mathrm{arg\,max}}_{i\geq 1} \pseudonorm{\varphi^{(i)}_h}[E] \; , \ \ \ 
  h\in\Zset \; .
\end{equation}
Provided $\pr(\pseudonorm{\bZ_h}[E]>0)>0$, there are almost surely infinitely many particles at
time $h$ with positive norm and a unique particle with the largest norm. This is because the random
variables $\Gamma_i$, $i\geq1$ have continuous distributions and
$\lim_{i\to\infty}\Gamma_i^{-1/\alpha}=0$ almost surely.  Therefore the $\mathrm{arg\,max}$
in~\eqref{eq:def-X-argmax} is unique and the random variable $i_h$ is well-defined.

\begin{theorem}
\label{theo:nu-to-X}
  Given a shift-invariant tail measure $\tailmeasure$ on $\setE^{\Zset}$, the $\setE$-valued time
  series $\bX$ defined by~\eqref{eq:def-X-argmax} is stationary and regularly varying on
  $\seqspace\setE$ with sequence $a_n=n^{1/\alpha}$ and tail measure $\tailmeasure$.
\end{theorem}

\begin{proof}

  We will use the mapping \Cref{lem:transfer-homogeneous-regvar}. We consider $\setF=\seqspace\setE$
  endowed with the metric $\dist[F]$ defined in~(\ref{eq:distF}).

  Define the subset $\mcn_0^\sharp(\setE)\subset \mcn_0(\setE)$ as the set of point measures that
  have exactly one largest point and consider the map $T:\mcn_0^\sharp(\setE)\to \setE$ that
  associate to such a point measure its largest point. By \Cref{lem:N0sharp}, $\mcn_0^\sharp(\setE)$
  is open and $T$ is continuous on $\mcn_0^\sharp(\setE)$. We extend $T$ to $\mcn_0(\setE)$ by
  setting the value $\zero{\setE}$ on $\mcn_0(\setE)\setminus\mcn_0^\sharp(\setE)$.

  Given a point measure $\pi\in\mcn_0(\setF)$ and $h\in\Zset$, we define $P_h(\pi)$ as the
  restriction to $\setE\setminus\{\zero\setE\}$ of the image of $\pi$ under the projection
  $\bx\mapsto\bx_h$. More precisely, if $\pi = \sum_{i=1}^\infty \delta_{\bx^{(i)}}$ with
  $\bx^{(i)}\in\setF$, then $P_h(\pi)$ is the point measure on $\setEzero$ with points $\bx_h^{(i)}$
  such that $\bx_h^{(i)}\ne\zero\setE$.  For the particle system $\Pi=\{\varphi^{(i)},i\geq 1\}$,
  \[
  P_h\Pi=\{\varphi_h^{(i)} : i\geq 1\;,\; \varphi_h^{(i)}\neq \zero\setE\}
  \]
  records the position at time $h\in\Zset$ of the non zero particles.  Using the
  representation~(\ref{eq:representation-Pi}), we also have
  \begin{align*}
      P_h\Pi=\{\Gamma_i^{-1/\alpha}\bZ_h^{(i)} : i\geq 1\;,\; \bZ_h^{(i)}\neq \zero\setE\} \; . 
  \end{align*}
  Since $\Pi$ is Poisson, $P_h\Pi$ is a Poisson point process on $\setE\setminus\{\bszero_\setE\}$ with intensity
  \[
  \mu(B) = \tailmeasure(\{ \bx: \bx_h\in B\}) = \tailmeasure(\{ \bx: \bx_0\in B\}) \; , 
  \]
  for all Borel measurable sets $B\subset\setE\setminus\{\zero\setE\}$.  The marginal measure $\mu$
  does not depend on $h\in\Zset$ because $\tailmeasure$ is shift-invariant. Moreover,
  $P_h\Pi\in \mcn_0^\sharp(E)$ almost surely since for $i\ne j$,
  $\pr(\Gamma_i^{-1/\alpha}\pseudonorm{\bZ_h^{(i)}}[E] =
  \Gamma_j^{-1/\alpha}\pseudonorm{\bZ_h^{(j)}}[E] ) = 0$.

  We now define the map $\mct$ on $\mcm_0(\setF)$ onto $\setF$ by
  \begin{align*}
    \mct(\pi) = \{T(P_h\pi), h\in\Zset\} \; . 
  \end{align*}
  The time series $\bX$ defined in (\ref{eq:def-X-argmax}) can be reexpressed in terms of the map
  $\mct$: $\bX = \mct(\Pi)$.  The stationarity of $\bX_h$  follows from the shift-invariance of $\tailmeasure$ since
  $\shift\bX=\mct(\shift \Pi) \equallaw \mct(\Pi)$ where
  $\shift\Pi=\{B\varphi^{(i)},i\geq 1\}\equallaw \Pi$.  The regular variation of $\bX$ will be
  obtained as a consequence of \Cref{lem:transfer-homogeneous-regvar}. By construction, $\mct$ is
  1-homogeneous, $\mct(\nullmeasure) = \zero\setF$ and we will check the following properties:
  \begin{enumerate}[(a),wide=0pt]
  \item \label{item:continuousatzero} the map $\mct$ is continuous at $\nullmeasure$; 
  \item \label{item:almostsurelycontinuous} the map $\mct$ is almost surely continuous with respect to  the distribution of $\Pi$.
  \end{enumerate}

  \begin{enumerate}[-,wide=0pt]
  \item To prove that $\mct$ is continuous at $\nullmeasure$, recall that the space $\setF$ is
    endowed with the distance defined in~(\ref{eq:distF}) and note that for
    $\pi=\sum_{i=1}^\infty \delta_{\bx^{(i)}}\in\mcn_0(\setF)$,
    \begin{align*}
      \dist[F](\zero\setF,\mct(\pi))  = \sum_{h\in\Zset} 2^{-|h|} \max_{i\geq1}(\pseudonorm{\bx_h^{(i)}}[E]\wedge1)
      \leq 3 \max_{i\geq1} \dist[F](\zero\setF,\bx^{(i)}) = 3 \largestpoint\pi[F] \; . 
    \end{align*}
    On the other hand, applying~(\ref{eq:minoration-prohorov-nullmeasure}), we obtain that if
    $ \rho(\nullmeasure,\pi) < 1/4$, then
    \begin{align*}
      \dist[F](\zero\setF,\mct(\pi)) \leq  12  \rho(\nullmeasure,\pi) \; .
    \end{align*}
    This proves~\ref{item:continuousatzero}.
  \item We now prove~\ref{item:almostsurelycontinuous}. By \Cref{lem:continuity-distF}, it suffices
    to prove that the projections $T_h = T\circ P_h$ are continuous for all $h$. Since
    $\pr(P_h\Pi\in\mcn_0^\sharp(\setE))=1$, this follows from the continuity of $T$ on
    $\mcn_0^\sharp(\setE)$ which is established in \Cref{lem:N0sharp}. 
  \end{enumerate}

  To conclude the proof, there only remains to prove that the tail measure of $\bX$ is
  $\tailmeasure$. By \Cref{lem:transfer-homogeneous-regvar} and \Cref{theo:RV-N0}, the tail measure
  of $\bX$ is $\mu^*\circ \mct^{-1}$, given 
   for $A\in\setFzero$ by
  \begin{align*}
    \mu^*\circ \mct^{-1}(A) = \int_\setF \ind{\mct(\delta_{\bx})\in A} \tailmeasure(\rmd \bx) \; .
  \end{align*}
  For $\bx=\{\bx_h,h \in\Zset\}\in\setF$, we have
  $\mct(\delta_{\bx})=\{T(\delta_{\bx_h}),h\in\Zset\} = \bx$ if $\bx\ne\zero\setF$ and
  $\mct(\zero\setF)=\mct(\nullmeasure) = \zero\setF$. Thus $\mu^*\circ \mct^{-1} = \tailmeasure$.
\end{proof}

The next two proposition state some interesting elementary properties of the process $\bX$ defined
by \eqref{eq:def-X-argmax}. They are strongly related to max-stability.  Let $g:\setF\to\setF$ be
the map defined by $g(\bx) = \{\pseudonorm{\bx_h}[E],h\in\Zset\}$. 
\begin{proposition}
  \label{prop:norm-max-stable}
  Consider the process $\bX$ defined by \eqref{eq:def-X-argmax}. Then the non negative time series
  $\{\pseudonorm{\bX_h}[E],h\in\Zset\}$ is max-stable with exponent measure $\tailmeasure\circ g^{-1}$. 
\end{proposition}

\begin{proof}
  The max stability follows from the representation
  $\pseudonorm{\bX_h}[E]=\sup_{i\geq1} \Gamma_i^{-1/\alpha} \pseudonorm{Z_i}[E]$ and the fact that
  $\tailmeasure\circ g^{-1}$ is the exponent measure is a consequence of the mapping theorem
  \Cref{lem:transfer-homogeneous-regvar}, since for a max-stable process, the tail measure and
  exponent measure are the same.
\end{proof}

In order to study further the stability property of the process $\bX$, we define the binary
operation $\odot$ defined on $\setE$ by
\[
\bx_1\odot \bx_2 =\left\{\begin{array}{ll} \bx_1 & \mbox{if $\pseudonorm{\bx_1}[E]\geq \pseudonorm{\bx_2}[E]$} \\ \bx_2 &
    \mbox{otherwise}\end{array}\right.\;,\quad \bx_1,\bx_2\in \setE \; .
\]
Note that the binary operation $\odot$ is associative, that is
$(\bx_1\odot \bx_2)\odot \bx_3=\bx_1\odot (\bx_2\odot \bx_3)$ for all $\bx_1,\bx_2,\bx_3\in \setE$. It is not
commutative since $\bx_1\odot \bx_2\neq \bx_2\odot \bx_1$ if $\bx_1$ and $\bx_2$ are distinct elements with the
same norm. However, elements with distinct norms do commute. More generally, if $\bx_1,\dots,\bx_n$ are
elements in $\setE$ such that exactly one element has maximal norm, $\bx^\ast$ say, then
$\bx_1\odot\cdots\odot \bx_n=\bx^\ast$ does not depend on the order of the $\bx_i$'s.

\begin{proposition}
  \label{prop:odot-stable}
  The process $\bX$ defined by \eqref{eq:def-X-argmax} admits the Lepage representation
  \begin{align}
    \label{eq:representation-x-odot}
    \bX \equallaw \odot_{i= 1}^\infty \Gamma_i^{-1/\alpha}\bZ^{(i)} \; ,
  \end{align}
  with $\{\Gamma_i,i\geq 1\}$ and $\{\bZ^{(i)},i\geq 1\}$ as in \eqref{eq:representation-Pi} and the
  operation $\odot$ is taken componentwise.  Furthermore, the process $\bX$ is stable with respect
  to the operation $\odot$ in the sense that,
  \begin{align}
    \label{eq:odot-stable}
    n^{-1/\alpha}\odot_{i=1}^n \bX^{(i)} \equallaw \bX \; , 
  \end{align}
  for every $n\geq1$, $\bX^{(1)},\ldots,\bX^{(n)}$ being independent copies of $\bX$.
\end{proposition}

\begin{proof}
  The representation (\ref{eq:representation-x-odot}) is simply a rewriting of the definition of the
  process $\bX$, that is
  \begin{align*}
    \odot_{i= 1}^\infty \Gamma_i^{-1/\alpha}\bZ^{(i)}  = \mct(\Pi) \; , 
  \end{align*}
  where $\Pi \sim \mathrm{PPP}(\tailmeasure)$. Let $n\geq1$, $\Pi_1,\dots,\Pi_n$ be \iid\ copies of
  $\Pi$ and  $\bX^{(1)},\ldots,\bX^{(n)}$ be independent copies of $\bX$.  Since $\mct$ is 1-homogeneous, we have 
  \begin{align*}
    n^{-1/\alpha}\odot_{i=1}^n \bX^{(i)} = \mct(n^{-1/\alpha} \Pi_1 \cup \cdots \cup n^{-1/\alpha} \Pi_n) \equallaw \mct(\Pi) \; , 
  \end{align*}
since $n^{-1/\alpha} \Pi_1 \cup \cdots \cup n^{-1/\alpha} \Pi_n \sim \mathrm{PPP}(\tailmeasure)$. 
\end{proof}

\subsection{Extremal indices and $m$-dependent approximation}\label{sec:extremal-indices}
The purpose of this section is to investigate more advanced properties of the process $\bX$ defined
by \eqref{eq:def-X-argmax} such as existence of extremal indices and $m$-dependent tail equivalent
approximations. Anti-clustering is also discussed in the next section. For the sake of generality, we
do not restrict our study to the process \eqref{eq:def-X-argmax} but rather consider a large class
of processes constructed on the Poisson particle system $\Pi\sim\mathrm{PPP}(\tailmeasure)$.

Let us first introduce the notion of extremal index that provides an insight in the dependence
structure of a stationary regularly varying time series. For a time series
$\bxi\in \mathrm{RV}_\alpha(\seqspace{[0,\infty)},(a_n),\tailmeasure)$, we compare the growth
rates of
\[
M_n=\max_{1\leq h\leq n} \xi_h \; , \quad\mbox{ and }\quad \widetilde M_n=\max_{1\leq h\leq n} \widetilde\xi_h \; , 
\]
where the random variables $\widetilde\xi_h$ are independent copies of $\xi_0$. Regular variation
and independence imply that $\widetilde{M}_n/a_n$ converges to a standard $\alpha$-Fr\'echet
distribution, that is
\begin{align*}
  \lim_{n\to\infty}\pr\left(a_n^{-1}\max_{1\leq h\leq n} \widetilde\xi_h\leq x\right) =
  \rme^{-x^{-\alpha}} \; , 
\end{align*}
for all $x>0$. Under assumptions discussed below, one can prove that
\begin{align*}
  \lim_{n\to\infty}\pr\left(a_n^{-1}\max_{1\leq h\leq n}  \xi_h\leq x\right) = \rme^{-\theta x^{-\alpha}} \; , 
\end{align*}
for $x>0$ where $\theta\in [0,1]$ is called the extremal index. If $\theta=0$, we have
$a_n^{-1}M_n\stackrel{P}\to 0$: the maximum has a slower growth rate in the dependent case. When
$\theta>0$, the maximum grows at rate $a_n$ as in the independent case. The extremal index can also
be defined as the limit, if it exists,
\begin{align}
  \label{eq:def-extermalindex}
  \theta = \lim_{n\to\infty} \log \frac{\pr\left(\max_{1\leq h\leq n} \xi_h\leq a_n\right)}{\pr(\xi\leq a_n )^n} \; .
\end{align}
In the abstract framework $\bX\in\mathrm{RV}(\seqspace\setE,(a_n),\tailmeasure)$, we consider, for
any $1$-homogeneous continuous $H:\setE\to[0,\infty)$, the extremal index (if it exists) of the non
negative time series $\{H(\bX_h),h\in\Zset\}$:
\begin{equation}
  \label{eq:theta_H}
  \theta_{H} = \lim_{n\to\infty}\log \frac{\pr\left(\max_{1\leq h\leq n} H(\bX_h)\leq a_n\right)}{\pr(H(\bX_0)\leq a_n )^n}.
\end{equation}
The homogeneity and continuity of $H$ ensure that
$H(\bX)\in\mathrm{RV}(\seqspace{[0,\infty)},\{a_n\},\tailmeasure\circ H^{-1})$, provided
$\tailmeasure\circ H^{-1}$ is not the null measure.

There exists a vast literature on the extremal index and several conditions have been introduced
that ensure the existence of a positive extremal index.  Building on
\cite{chernick:hsing:mccormick:1991} and using the tail measure through the tail process introduced
by \cite{basrak:segers:2009}, we will only consider here a condition based on $m$-dependent tail
equivalent approximations.  An $\setE$-valued time series $\bX$ is called $m$-dependent if the
$\sigma$-fields $\sigma(\bX_h,h\leq h_0)$ and $\sigma(\bX_h,h\geq h_0+m+1)$ are independent for all
$h_0\in\Zset$. In particular, a stationary $0$-dependent time series is a series of
independent and identically distributed random variables.
\begin{definition}
  \label{def:tail-equ-m-dep-approx}
  A  process $\bX$ is said to have a tail equivalent approximation if there exists
  a sequence of processes $\{\bX^{(m)}, m\geq 1\}$ such that:
    \begin{align}
      \label{eq:tail-equiv}
      \lim_{m\to\infty} \limsup_{n\to\infty} n\pr(\dist[\setE](\bX_h/a_n,\bX_h^{(m)}/a_n)>\epsilon) = 0 \; .
    \end{align}
\end{definition}

The relationship between $m$-dependent tail equivalent approximation and existence of an extremal
index is made clear in the following theorem. Since the extremal index is essentially defined for
non-negative time series, we focus on that case.
\begin{lemma}
  \label{theo:m-dep-ext-index-general}
  Let $\bX^{(m)}\in\mathrm{RV}(\seqspace{[0,\infty)},\{a_n\},\tailmeasure^{(m)})$ be stationary and
  $m$-dependent. Then, $\bX^{(m)}$ has a positive extremal index equal to the maximal index
  $\theta_{\tau_0}^{(m)}$ of $\tailmeasure^{(m)}$ associated to the map $\tau_0$ defined on
  $[0,\infty)^\Zset$ by $\tau(\bx)=x_0$.  If moreover $\bX^{(m)}$ is a tail equivalent approximation
  of a non negative time series $\bX$ and if the limit $\lim_{m\to\infty}\theta_{\tau_0}^{(m)}$
  exists, then it is the extremal index of $\bX$.
\end{lemma}

\begin{proof}
  Since an $m$-dependent sequence is $\alpha$ mixing with arbitrary fast rate,
  the existence of the extremal index $\theta$ is proved by \cite[Theorem~4.5]{basrak:segers:2009}
  and is given by $\theta^{(m)} = \pr(\max_{i\geq 1} Y_i^{(m)}\leq1)$. Thus $\theta=\theta_{\tau_0}$
  by \Cref{rem:candidate}.  The second statement is a consequence of
  \cite[Proposition~1.4]{chernick:hsing:mccormick:1991}.
\end{proof}

Based on this result, we now prove the existence of the extremal index $\theta_H$ the process $\bX$
considered in \eqref{eq:def-X-argmax}.  The process $\bX$ is defined by means of the stationary
$\mcn_0(\setE)$-valued sequence $\bP=\{P_h(\Pi),h\in\Zset\}$ and the map $T$ introduced in the proof
of \Cref{theo:nu-to-X} but the specific form of $T$ is irrelevant and only $1$-homogeneity and
continuity are needed. Therefore we will first prove that the stationary sequence~$\bP$ admits an
$m$-dependent tail equivalent approximation and then obtain the extremal index of time series
derived from~$\bP$.
\begin{theorem}
  \label{theo:m-dep-app}
  Let $\tailmeasure$ be a tail measure on $\seqspace\setE$ and $\Pi\sim\mathrm{PPP}(\tailmeasure)$
  be the associated particle process. Consider the stationary $\mcn_0(\setE)$-valued process
  $\bP=\{P_h(\Pi), h\in \Zset\}$.  If $\tailmeasure$ has a dissipative representation
  \eqref{eq:diss-representation}, then $\bP$ has an $m$-dependent tail-equivalent approximation.
\end{theorem}
\begin{proof}
  Note first that $\Pi$ can be expressed as
    \begin{align}
      \label{eq:representation-pi}
      \Pi = \sum_{i\geq1} \delta_{\Gamma_i^{-1/\alpha} \shift^{T_i} \tbZ^{(i)}}  \; , 
    \end{align}
    where $\delta_{\Gamma_i^{-1/\alpha}}$ is a Poisson point process on $(0,\infty)$ with mean
    measure $\nu_\alpha$, $\tbZ^{(i)}$ are \iid\ copies of the process $\tbZ$
    in~(\ref{eq:diss-representation}), $\shift$ is the shift operator and
    $\sum_{i=1}^\infty \delta_{T_i}$ is a Poisson point process on $\Zset$ with mean measure the
    counting measure on $\Zset$, independent of everything else. Indeed, it suffices to check that
    the mean measure of the point process on the right hand side of~(\ref{eq:representation-pi}) is
    $\tailmeasure$. This follows from (\ref{eq:diss-representation}).

 We now define the $m$-dependent approximation $\bP^{(m)}$ of $\bP$. For $m\geq1$, define
    \begin{align*}
      \bP_h^{(m)} & = \sum_{i\geq1} \delta_{\Gamma_i^{-1/\alpha} \ \tbZ_{h-T_i}^{(i)}} \ind{|h-T_i|\leq m} \; . 
    \end{align*}
 We must now check the tail equivalence condition (\ref{eq:tail-equiv}). By stationarity, it
    suffices to check it for $h=0$. That is we must prove that for all $\epsilon>0$, 
    \begin{align}
      \label{eq:tail-equiv-bP0}
      \lim_{m\to\infty} \limsup_{n\to\infty} n\pr(\rho(a_n^{-1}\bP_0,a_n^{-1}\bP_0^{(m)})>\epsilon) = 0  \; .
    \end{align}
    Set $R_m = \bigvee_{i=1}^\infty  \Gamma_i^{-1/\alpha} \ \pseudonorm{\tbZ_{-T_i}^{(i)}}[E]\ind{|T_i|>m}$. For
    $r>a_n^{-1}R_m$, $a_n^{-1}\bP_0^{(m)}$ and $a_n^{-1} \bP_0$ have the same points on $B_r^c$. Therefore
    \begin{align*}
      \rho(a_n^{-1}\bP_0,a_n^{-1}\bP_0^{(m)}) = \int_0^{a_nR_m} (\rho_r(a_n^{-1}\bP_0,a_n^{-1}\bP_0^{(m)})\wedge1) \rme^{-r} \rmd r \leq a_n R_m \; . 
    \end{align*}
    Thus~(\ref{eq:tail-equiv-bP0}) will be obtained  as a consequence of
    \begin{align}
      \label{eq:tail-equiv-Rm}
      \lim_{m\to\infty} \limsup_{n\to\infty} n\pr(R_m>a_n\epsilon) = 0  \; .
    \end{align}
 To prove (\ref{eq:tail-equiv-Rm}), note that for non negative random variables $Z_i,i\geq1$,
    since $\Gamma_1^{-1/\alpha}$ has a Fr\'echet distribution, we have
    \begin{align*}
      \pr\left(\bigvee_{i=1}^\infty \Gamma_i^{-1/\alpha} Z_i>x \right) 
      & \leq \sum_{i=1}^\infty \pr\left(\Gamma_i^{-1/\alpha} Z_i>x \right) \\
      & \leq \sum_{i=1}^\infty \pr\left(\Gamma_1^{-1/\alpha} Z_i>x \right)  = \sum_{i=1}^\infty (1-\rme^{-x^{-\alpha}\esp[Z_i^\alpha]}) \; .
    \end{align*}
    Therefore, if $\sum_{i=1}^\infty \esp[Z_i^\alpha]<\infty$, we obtain by dominated convergence 
    \begin{align*}
      \limsup_{x\to\infty} x^\alpha \pr\left(\bigvee_{i=1}^\infty \Gamma_i^{-1/\alpha} Z_i>x \right) \leq \sum_{i=1}^\infty \esp[Z_i^\alpha] \; . 
    \end{align*}
    Applying this bound to $R_{m}$, we obtain by dominated convergence theorem 
    \begin{align*}
      \lim_{m\to\infty}  \limsup_{n\to\infty} n\pr(R_m>a_n\epsilon) \leq       
      \lim_{m\to\infty} \sum_{i=1}^\infty \esp[\pseudonorm{\bZ_0}[E]^\alpha\ind{|T_i|>m}] = 0 \; .
    \end{align*}
    This proves~(\ref{eq:tail-equiv-Rm}).
\end{proof}

To a function $H:\mcn_0(\setE)\to[0,\inftyà)$ we associate the function
$\widehat{H}: \seqspace{\setE}\to \seqspace{[0,\infty)}$ defined by
$\widehat{H}(\bx)=\{H(\delta_{\bx_h}),h\in\Zset\}$, $\bx\in\seqspace\setE$.

\begin{corollary}
  \label{theo:ext-index-particlesystem}
  Under the assumptions of \Cref{theo:m-dep-app}, let $H:\mcn_0(\setE)\to [0,\infty)$ be a
  Lipschitz continuous $1$-homogeneous function such that $\tailmeasure(\{\hat{H}(\bx)>1\})=1$.
  Then the time series $\bX_H=\{H\circ P_h(\Pi),h\in\Zset\}$ is in
  $\mathrm{RV}(\seqspace{[0,\infty)},\{n^{1/\alpha}\},\tailmeasure_H)$ with 
  $\tailmeasure_H = \tailmeasure\circ \widehat H^{-1}$, has an $m$-dependent tail
  equivalent approximation and an extremal index equal to the maximal index $\theta_{\tau}$
  associated to $\tailmeasure$ and the map $\tau$ defined on $\seqspace\setE$ by
  $\tau(\bx) = H(\delta_{\bx_0})$.
\end{corollary}

\begin{proof}
  We will apply \Cref{theo:m-dep-ext-index-general,theo:m-dep-app}. Let $\bP^{(m)}$ be the
  $m$-dependent approximation of $\bP$ defined in the proof of \Cref{theo:m-dep-app}. Then the time
  series $\bX^{(m)}$ defined by $\bX_h^{(m)} = H\circ P_h(\Pi)$, $h\in\Zset$ is $m$ dependent and
  regularly varying by \Cref{lem:transfer-homogeneous-regvar}.  By \Cref{theo:m-dep-ext-index-general},  its extremal index $\theta^{(m)}$ is
  given by
  \begin{align*}
    \theta^{(m)} = \frac{\esp\left[ \max_{|h|\leq m} \bar{H}^\alpha(\widetilde{\bZ}_h)\right]} 
    {\esp\left[\sum_{|h|\leq m} \bar{H}^\alpha(\widetilde{\bZ}_h)\right]} \; . 
  \end{align*}
The tail equivalence condition (\ref{eq:tail-equiv}) holds by the Lipschitz property of $H$. Thus
  the sequence $\{\bX^{(m)}\}$ is a tail equivalent approximation of $\bX$ and we can apply
  \Cref{theo:m-dep-ext-index-general} which proves (by application of the dominated convergence
  theorem) that the extremal index of $\bX$ is
  \begin{align*}
    \theta = \lim_{m\to\infty} \theta^{(m)} = \frac{\esp\left[ \max_{h\in\Zset} \bar{H}^\alpha(\widetilde{\bZ}_h)\right]} 
    {\esp\left[\sum_{h\in\Zset} \bar{H}^\alpha(\widetilde{\bZ}_h)\right]} \; .
  \end{align*}
\end{proof}

\subsection{The anti-clustering condition}
In the literature of time series and extremal index, the anti-clustering condition introduced by \cite{davis:hsing:1995} plays quite an
important role, see e.g.,  \cite{janssen2018eigenvalues}, \cite{MR3547739}, \cite{MR3535966}, \cite{MR3025708}. Let us first define the notion of anti-clustering for a stationary
regularly varying sequence.

\begin{definition}
  A stationary time series $\bX\in\mathrm{RV}(\seqspace{\setE},(a_n),\tailmeasure)$ satisfies the
  anti-clusering condition if there exists an intermediate sequence $r_n\to\infty$, $r_n/n\to 0$,
  such that
  \begin{equation}
    \label{eq:AC-abstract}
    \lim_{m\to\infty} \limsup_{n\to\infty} \pr\left( \max_{m\leq |h|\leq r_n} \dist[\setE](\bX_t/a_n,\zero\setE)>u  
      \mid  \dist[\setE](\bX_0/a_n,\zero\setE)>u \right) = 0 \; . 
  \end{equation}
\end{definition}
When $\setE=\Rset$, we retrieve the classical anti-clustering condition of
\cite[Condition~2.8]{davis:hsing:1995}. 
Although we have not used anti-clustering in our analysis of extremal ind
(\Cref{theo:ext-index-particlesystem}), we show below that, for the class of processes
considered, anti-clustering is equivalent to the existence of a dissipative representation for
$\tailmeasure$. This suggests that assuming the existence of a dissipative representation for
$\tailmeasure$ in \Cref{theo:ext-index-particlesystem} is not a too strong condition.
\begin{theorem}
  \label{theo:anticlustering}
  The following statements are equivalent:
  \begin{enumerate}[(i)]
  \item \label{item:dissipative2} $\tailmeasure$ has a dissipative representation \eqref{eq:diss-representation};
  \item \label{item:PI} the process $(P_h(\Pi))_{h\in\Zset}$ satisfies the anti-clustering condition in
    $(\mcn_0(\setE))^\Zset$ in $[0,\infty)^\Zset$;
  \item \label{item:H} for all $H$ as in \Cref{theo:ext-index-particlesystem}, $\bX_H$ satisfies the
    anti-clustering condition in $[0,\infty)^\Zset$;
  \item \label{item:maxstab} the max-stable process $\{\largestpoint{P_h(\Pi)}[\setE],h\in\Zset\}$ satisfies the
    anti-clustering condition $[0,\infty)^\Zset$.
\end{enumerate}
\end{theorem}
\begin{proof}
Since the process $\{\largestpoint{P_h(\Pi)}[\setE],h\in\Zset\}$ is max-stable, the equivalence
  between~\ref{item:dissipative2} and~\ref{item:maxstab} is proved in
  \cite[Theorem~2.1]{debicki:hashorva:2016}. The implication \ref{item:PI}~$\Rightarrow$~\ref{item:H} is a consequence of the Lipshitz property
  of $H$; the implication \ref{item:H}~$\Rightarrow$~\ref{item:maxstab} is trivial since the map
  $\largestpoint{\cdot}[\setE]$ satisfies the condition of
  \Cref{theo:ext-index-particlesystem}. Conversely, \ref{item:maxstab} implies \ref{item:PI} since
  $ \frac{1}2 (\largestpoint{P_h(\Pi)}[\setE]\wedge1) \leq
  \dist[\mcn_0(\setE)](a_n^{-1}P_h(\Pi),\nullmeasure)\leq\largestpoint{P_h(\Pi)}[\setE]$
  by~(\ref{eq:majoration-prohorov-nullmeasure}) and~(\ref{eq:minoration-prohorov-nullmeasure}). 
\end{proof}

\appendix
\section{Convergence in $\mcm_0(\mcn_0(\setF))$}

For $\mu \in \mcm_0(\mcn_0(\setF))$, we denote by $\mathcal{B}_\mu$ the set of Borel sets
$B\subset \mcn_0(\setF)$ that are bounded away from zero and such that $\mu(\pi(\partial B)>0)=0$,
with $\partial B$ the boundary of $B$.

\begin{theorem}
  \label{theo:Laplace}
  Let $\mu,\mu_1,\mu_2,\ldots \in \mcm_0(\mcn_0(\setF))$.  The following statements are equivalent:
  \begin{enumerate}[(i)]
  \item \label{item:convmzero} $\mu_n\convmzero{\mcn_0(\setF)}\mu$ as $n\to\infty$.
  \item \label{item:convfidi} $\mu_n\stackrel{fidi}\longrightarrow\mu$ as $n\to\infty$, in the sense that
    \[
    \mu_n(\pi(A_i)=m_i,\ 1\leq i\leq k) \to \mu(\pi(A_i)=m_i,\ 1\leq i\leq k)\quad \mbox{as
      $n\to\infty$}
    \]
    for all $k\geq 1$, $(m_1,\ldots,m_k)\in \mathbb{N}^k\setminus\{0\}$ and
    $A_1,\ldots,A_k\in \mathcal{B}_\mu$.
  \item \label{item:convlaplace} for all bounded continuous $f:\setF\to [0,\infty)$ vanishing on a
    neighborhood of $\zero\setF$,
    \[
    \int_{\mcn_0(\setF)}\left(1-\rme^{-\pi(f)}\right)\mu_n(\rmd \pi) \longrightarrow
    \int_{\mcn_0(\setF)}\left(1-\rme^{-\pi(f)}\right)\mu(\rmd \pi)\quad \mbox{as $n\to\infty$}
    \]
    with $\pi(f)=\int_\setF f(x)\pi(\rmd x)$.
  \end{enumerate}
\end{theorem}
This theorem is   similar to the characterization of weak convergence of probablity measure on
$\mcn_0(\setF)$ in terms of their finite dimensional distributions and Laplace functional by
\cite[Theorem 3.10 and Corollary 3.11]{zhao:2016}.  We consider here $M_0$-convergence of measures
with possibly infinite total mass, so that we exclude in~\ref{item:convfidi} the event
$\{\pi(A_i)=0,\ 1\leq i\leq k\}$ that may have infinite mass and we use in~\ref{item:convlaplace} a
modified Laplace transform with $1-\rme^{-\pi(f)}$ instead of $\rme^{-\pi(f)}$ so as to ensure that
the integrals are finite.

\begin{proof}
  We begin with some notation and preliminaries that will be used throughout the proofs below. We
  denote by $B_{\setF,r}^c$ (resp.  $B_{\mcn_0(\setF),r}^c$) the complement of the ball with center
  $0$ and radius $r>0$ in $\setF$
  (resp. $\mcn_0(\setF)$). The bounds~(\ref{eq:majoration-prohorov-nullmeasure})
  and~(\ref{eq:minoration-prohorov-nullmeasure}), imply that  for $r\leq 1$, 
  \begin{equation}
    \label{eq:inclusion}
    B_{\mcn_0(\setF),r}^c \subset \{\pi(B_{\setF,r}^c)>0\} \subset B_{\mcn_0(\setF),r/4}^c.
  \end{equation}

  Let $\mu\in\mcm(\mcn_0(\setF))$ be fixed and consider a sequence $r_i\downarrow 0$ such that
  $\mu(\partial B_{\mcn_0(\setF),r_i}^c)=\mu(\pi(\partial B_{\setF,r_i}^c)>0)=0$ for all $i\geq
  1$.
  By  \cite[Theorem~2.2]{hult:lindskog:2006}, the $M_0$-convergence
  $\mu_n\stackrel{M_0}\longrightarrow \mu$ is equivalent to the weak convergence
  $\mu_n^{(r_i)}\stackrel{w}\longrightarrow \mu^{(r_i)}$ for all $i\geq 1$, where $\mu_n^{(r)}$
  (resp. $\mu^{(r)}$) denotes the restriction of $\mu_n$ (resp. $\mu$) to
  $B_{\mcn_0(\setF),r}^c$. By the inclusion \eqref{eq:inclusion}, this is also equivalent to the
  weak convergence $\tilde\mu_n^{(r_i)}\stackrel{w}\longrightarrow \tilde\mu^{(r_i)}$ for all
  $i\geq 1$, where $\tilde\mu_n^{(r)}$ (resp. $\tilde\mu^{(r)}$) denotes the restriction of $\mu_n$
  (resp. $\mu$) to $\{\pi(B_{\setF,r_i}^c)>0\}$. The restriction $\tilde\mu_n^{(r_i)}$ will be
  useful because they behave well with respect to finite dimensional distributions.

  The weak convergence $\mu_n^{(r_i)}\stackrel{w}\longrightarrow \mu^{(r_i)}$ (or
  $\tilde\mu_n^{(r_i)}\stackrel{w}\longrightarrow \tilde\mu^{(r_i)}$) of finite measures can be
  characterized as in \cite[Theorem 3.10 and Corollary 3.11]{zhao:2016} by the weak convergence of
  finite dimensional distributions or pointwise convergence of Laplace functionals. Note that the
  result and proof in \cite{zhao:2016} are given for weak convergence of probability measures only,
  but they are easily extended to finite measures since weak convergence is then equivalent to
  convergence of the total mass together with weak convergence of the normalized measures.

  \textit{Proof of \ref{item:convmzero}~$\Rightarrow$~\ref{item:convfidi}.} From the
  preliminary discussion, the $M_0$-convergence $\mu_n\stackrel{M_0}\longrightarrow \mu$ implies,
  for all $i\geq 1$, the convergence of
  $\mu_n(B_{\mcn_0(\setF),r_i}^c)\to \mu(B_{\mcn_0(\setF),r_i}^c)$ and the weak convergence of the
  finite dimensional distributions $\mu_n^{(r_i)}\stackrel{fidi}\longrightarrow \mu^{(r_i)}$.  This
  entails the convergence of finite dimensional distributions in the sense of $ii)$ because any set
  $A\in\mathcal{B}_\mu$ is bounded away from zero and hence in $B_{\mcn_0(\setF),r_i}^c$ for $r_i$
  small enough,
  
  \textit{Proof of \ref{item:convfidi}~$\Rightarrow$~\ref{item:convlaplace}.} It is enough to prove
  that \ref{item:convfidi} implies weak convergence of the finite dimensional distributions
  $\tilde\mu_n^{(r_i)}\stackrel{fidi}\longrightarrow \tilde\mu^{(r_i)}$ for all $r_i\geq 1$. Let
  $k\geq 1$, $A_1,\ldots,A_k\in\mathcal{B}_\mu$ and $m_1,\ldots,m_k\geq 0$. Setting
  $A_0=B_{\setF,r_i}^c\in \mathcal{B}_\mu$, we have
  \begin{equation}
    \label{eq:fidi1}
    \tilde\mu_n^{(r_i)}(\pi(A_j)=m_j,\ 1\leq j\leq k)=\tilde\mu_n(\pi(A_0)>0,\ \pi(A_j)=m_j,\ 1\leq j\leq k)
  \end{equation}
  and \ref{item:convfidi} implies convergence to
  \begin{equation}
    \label{eq:fidi2}
    \tilde\mu^{(r_i)}(\pi(A_j)=m_j,\ 1\leq j\leq k)=\tilde\mu(\pi(A_0)>0,\ \pi(A_j)=m_j,\ 1\leq j\leq k).
  \end{equation}
  This proves $\tilde\mu_n^{(r_i)}\stackrel{fidi}\longrightarrow \tilde\mu^{(r_i)}$ and
  $\mu_n\stackrel{M_0}\longrightarrow \mu$.

  \medskip \textit{Proof of \ref{item:convlaplace}~$\Rightarrow$~\ref{item:convmzero}.} We prove
  that \ref{item:convlaplace} implies that, for all $i\geq 1$, the measures $\tilde\mu_n^{(r_i)}$,
  $\tilde\mu^{(r_i)}$ have finite total mass and converge weakly
  $\tilde\mu_n^{(r_i)}\stackrel{w}\longrightarrow \tilde\mu^{(r_i)}$ as $n\to\infty$.  We first
  prove convergence of the total mass
\begin{equation}
  \label{eq:fidi4}
  \tilde\mu_n^{(r_i)}(\mcn_0(\setF))=\mu_n(\pi(B_{\setF,r_i}^c)>0)\to \tilde\mu^{(r_i)}(\mcn_0(\setF))=\mu(\pi(B_{\setF,r_i}^c)>0).
\end{equation}
Consider approximating functions $h_l^+(x)\downarrow \ind{x\in \mathrm{cl}B_{\setF,r_i}^c}$ and
$h_l^-(x)\uparrow \ind{x\in \mathrm{int}B_{\setF,r_i}^c}$ that are continuous with values in $[0,1]$
and vanish on a neighborhood of $0_\setF$.  The notation $\mathrm{cl}$ and $\mathrm{int}$ stands for
the closure and interior of the set respectively.

\[
\int_{\mcn_0(\setF)} \left(1-\rme^{-t\pi(h_l^-)}\right)\mu_n(\rmd\pi) \leq \int_{\mcn_0(\setF)}
\left(1-\rme^{-t\pi(B_{\setF,r_i}^c)}\right)\mu_n(\rmd\pi) \leq \int_{\mcn_0(\setF)}
\left(1-\rme^{-t\pi(h_l^+)}\right)\mu_n(\rmd\pi).
\]
The left and right hand sides in the previous inequalities converge and hence are bounded uniformly
in $n\geq 1$. Furthermore, since $\pi(B_{\setF,r_i}^c)$ takes values in $\{0,1,2,\ldots\}$, the
quantity
\[
\int_{\mcn_0(\setF)} \left(1-\rme^{-t\pi(B_{\setF,r_i}^c)}\right)\mu_n(\rmd\pi)=\sum_{m\geq 1} (1-\rme^{-tm})\mu_n\left(\pi(B_{\setF,r_i}^c)=m\right)
\]
satisfies
\[
(1-\rme^{-t})\mu_n\left(\pi(B_{\setF,r_i}^c)>0\right)\leq \int_{\mcn_0(\setF)}
\left(1-\rme^{-t\pi(B_{\setF,r_i}^c)}\right)\mu_n(\rmd\pi)\leq
\mu_n\left(\pi(B_{\setF,r_i}^c)>0\right).
\]
We deduce $\sup_{n\geq 1} \mu_n\left(\pi(B_{\setF,r_i}^c)>0\right)<\infty$ and this holds for all
$i\geq 1$.  Then, letting $n\to\infty$ in
\[
\int_{\mcn_0(\setF)} \left(1-\rme^{-t\pi(h_l^-)}\right)\mu_n(\rmd\pi)\leq
\mu_n\left(\pi(B_{\setF,r_i}^c)>0\right) \leq
(1-\rme^{-t})^{-1}\int_{\mcn_0(\setF)}\left(1-\rme^{-t\pi(h_l^+)}\right)\mu_n(\rmd\pi),
\] 
we get
\begin{align*}
  & \liminf_{n\to\infty} \mu_n(\pi(B_{\setF,r_i}^c)>0)\geq \int_{\mcn_0(\setF)}\left(1-\rme^{-t\pi(h_l^-)}\right)\mu(\rmd\pi), \\
  & \limsup_{n\to\infty} \mu_n(\pi(B_{\setF,r_i}^c)>0)\leq(1-\rme^{-t})^{-1} \int_{\mcn_0(\setF)}\left(1-\rme^{-t\pi(h_l^+)}\right)\mu(\rmd\pi).
\end{align*}
Letting $l\to\infty$ and $t\to\infty$, monotone convergence  entails that the right hand side in the
last two inequalities converge to $\mu(\pi(\mathrm{cl}B_{\setF,r_i}^c)>0)$ and
$\mu(\pi(\mathrm{int}B_{\setF,r_i}^c)>0)$ respectively. These two quantities are equal because we
have chose $r_i$ such that $B_{\setF,r_i}^c\in\mathcal{B}_\mu$, that is
$\mu(\pi(\partial B_{\setF,r_i}^c)>0)=0$. Consequently
$\mu_n(\pi(B_{\setF,r_i}^c)>0)\to \mu(\pi(B_{\setF,r_i}^c)>0)$ as $n\to\infty$, proving
\eqref{eq:fidi4}.

To prove that  $\tilde\mu_n^{(r_i)}\stackrel{w}\longrightarrow \tilde\mu^{(r_i)}$
using the Laplace functional, it is enough to prove
\[
\lim_{n\to \infty}\int_{\mcn_0(\setF)} \rme^{-\pi(f)} \tilde\mu_n^{(r_i)}(\rmd\pi)= \int_{\mcn_0(\setF)} \rme^{-\pi(f)} \tilde\mu{(r_i)}(\rmd\pi)
\]
for all bounded continuous $f:\setF\to [0,\infty)$ vanishing on a neighborhood of $\zero\setF$. In view
of equation \eqref{eq:fidi4}, this is equivalent to
\[
\lim_{n\to \infty} \int_{\mcn_0(\setF)}\left(1- \rme^{-\pi(f)}\right) \tilde\mu_n^{(r_i)}(\rmd\pi)=
\int_{\mcn_0(\setF)} \left(1- \rme^{-\pi(f)}\right) \tilde\mu{(r_i)}(\rmd\pi).
\]
The proof is similar to that of Equation \eqref{eq:fidi4} where the measures $\mu_n(\rmd\pi)$ and
$\mu(\rmd\pi)$ are replaced throughout the proof by $\left(1-\rme^{-\pi(f)}\right)\mu_n(\rmd\pi)$
and $\left(1-\rme^{-\pi(f)}\right)\mu(\rmd\pi)$ respectively. Details are left to the reader for the
sake of brevity. It is useful to note that
\[
\left(1-\rme^{-t\pi(h^\pm_l)}\right)\Big(1-\rme^{-\pi(f)}\Big)=\left(1-\rme^{-\pi(th^\pm_l)}\right)
+\Big(1-\rme^{-\pi(f)}\Big)-\left(1-\rme^{-\pi(f+th^\pm_l)}\right)
\]
so that \ref{item:convlaplace} allows to deal with the limit of the integrals as $n\to\infty$.
\end{proof}

\section{Lemmas for the proof of Theorem 3.7}
The Prohorov distance $\varrho_\setE$ between two bounded measures $\mu,\nu$ on a Borel space $\setE$ is defined by
\cite[Section~A2.5]{daley:vere-jones:bookvolI}
\begin{align}
  \label{eq:def-prohorov-mesures-finies}
  \varrho_\setE(\mu,\nu) = \inf\{\epsilon\geq0: \mu(A) \leq \nu(A^\epsilon) + \epsilon \; , \  \nu(A) \leq \mu(A^\epsilon) + \epsilon 
  \mbox{ for all closed sets $A$}\}
\end{align}

\begin{lemma}
  \label{lem:prohorov-point-measures}
  Let $\mu,\nu$ be two point measures on a metric space $(\setE,d)$. Then $\varrho_{\setE}(\mu,\nu)\geq
  |\mu(\setE)-\nu(\setE)|$. If $\bx,\by\in \setE$, then $\varrho_{\setE}(\delta_{\bx},\delta_{\by}) \leq d(\bx,\by)\wedge1$.
\end{lemma}

\begin{proof}
  Assume for instance that $\mu(\setE)=n$ and $\nu(\setE)=k$ with $k<n$. Let $\bx_1,\dots,\bx_n$ be the points
  of $\mu$ and let $A = \{\bx_1,\dots,\bx_n\}$. Then $A$ is closed, $\mu(A)=n$ and for all $\epsilon>0$,
  $\nu(A^\epsilon) \leq k$. This proves that $\varrho_\setE(\mu,\nu)\geq n-k$. The second statement is in
  \cite[Section~11.3~p.394]{dudley:2002}.
\end{proof}

Recall from \Cref{sec:construction} the definition of the set $\mcn_0^\sharp(\setE)$ and the map $T$
and the definition of the metric $\rho$
in~(\ref{eq:def-dist-rho}). 

\begin{lemma}
  \label{lem:N0sharp}
  The subset $\mcn_0^\sharp(\setE)$ is open in $\mcn_0(\setE)$ and the map $T:\mcn_0(\setE)\to\setE$
  is continuous on $\mcn_0^\sharp(\setE)$.
\end{lemma}

\begin{proof}
  Let $\pi\in\mcn_0^\sharp(\setE)$ and $m = \largestpoint{\pi}[E]$.  Then there exists $\eta>0$ such
  that $\pi$ has exactly one point in $B(T(\pi),\eta)$.  A point measure $\pi'\in\mcn_0(\setE)$ has either zero
  point or 1 or at least two points in $B_{m-\eta}^c$. By \Cref{lem:prohorov-point-measures}, in the
  first and last cases, $\rho_r(\pi,\pi') \wedge1 =1$, hence
  \begin{align*}
    \rho(\pi,\pi') \geq  \int_{m-\eta}^\infty \rme^{-r}   \rmd r \geq \rme^{-m} \; .
  \end{align*}
  Thus, if $\rho(\pi,\pi') < \rme^{-m}$, then $\pi'$ has exactly one point in $B_{m-\eta}^c$, which
  is therefore its single largest point and $\pi'\in\mcn_0^\sharp(\setE)$. This proves that
  $\mcn_0^\sharp(\setE)$ is open. By \Cref{lem:prohorov-point-measures} again, for $r>m-\eta$, we
  have $\rho_r(\pi,\pi') \wedge1 =\dist[E](T(\pi),T(\pi'))\wedge1$ thus
  \begin{align*}
    \rho(\pi,\pi') \geq  \int_{m-\eta}^\infty \rme^{-r} (\dist[E](T(\pi),T(\pi'))\wedge1)  \rmd r 
    \geq  (\dist[E](T(\pi),T(\pi'))\wedge1) \rme^{-m} \; .
  \end{align*}
  This proves that $T$ is continuous at $\pi$.
\end{proof}

\begin{lemma}
  \label{lem:continuity-distF}
  Let $(\setS,\dist[S])$ be a metric space and $g:\setS\to\setF=\seqspace\setE$. Then $S$ is continuous
  with respect to the distance $\dist[F]$ defined in~(\ref{eq:distF}) if and only if
  $g_j:\setS\to\setE$ defined by $g_j(\bs) = (g(\bs))_j$ for $\bs\in\setS$ is continuous for all
  $j\in\Zset$.
\end{lemma}

\begin{proof}
  The direct implication is trivial. We prove the converse. Assume that $g_j$ is continuous for all
  $j$. Fix $s_0\in\setS$, $\epsilon\in(0,1)$ and choose $K$ such that $2^{-K} \leq \epsilon/4$.  By
  assumption, there exists $\eta$ (which depends on $\epsilon$ and $K$) such that for all
  $j\in\{-K,\dots,K\}$ and $\bs\in\setS$ such $\dist[S](\bs_0,\bs)\leq \eta$,
  $\dist[E](g_j(\bs_0),g_j(\bs))\leq\epsilon/2$. This yields
  \begin{align*}
    \dist[F](g(\bs_0),g(\bs)) = \sum_{j\in\Zset} 2^{-|j|} \dist[E](g_j(\bs_0),g_j(\bs))\wedge1 
    \leq \frac{\epsilon}6 \sum_{|j|\leq K} 2^{-|j|} + \sum_{|j|> K} 2^{-|j|} \leq \epsilon \; .
  \end{align*}

\end{proof}

\paragraph{Acknowledgement} 
 Two anonymous referees and the associate editors are acknowledged for various suggestions that have improved the paper; in particular the introduction of the Hopf decomposition in Section 2.5 was suggested by a referee.

\bibliography{bib}

\end{document}